\documentclass[8.3pt,regno]{amsart}

\usepackage{amssymb,amsmath,amsthm,latexsym}
\usepackage{color}
\usepackage{graphicx}
\usepackage[abs]{overpic}
\usepackage{amssymb}

\newtheorem*{proposition*}{Proposition}
\newtheorem*{theorem*}{Theorem}
\newtheorem{theorem}{Theorem}

\newtheorem{defi}{Definition}[section]

\newtheorem{prop}[defi]{Proposition}

\newtheorem{cor}{Corollary}

\newtheorem{cl}[defi]{Claim}

\newtheorem{remark}[defi]{Remark}

\newtheorem{lemma}[defi]{Lemma}

\usepackage{amsfonts}
\usepackage{amssymb}
\usepackage{color}

\numberwithin{equation}{section}
\DeclareMathOperator{\diff}{Diff}

\def\bB{\mathbb{B}}

\def\bfs{{\mathbf{s}}}

\def\fs{{\mathfrak{S}}}

\def\h{\mathrm{ht}}
\def\q{\mathrm{qt}}

\newcommand{\eqdef}{\stackrel{\scriptscriptstyle\rm def}{=}}

\title[Non-transverse heterodimenional cycles]{ H\'enon-like families and blender-horseshoes at non-transverse heterodimensional cycles}
\author[L. J. D\'iaz and S. A. P\'erez ]{Lorenzo J. D\'iaz and Sebasti\'an A. P\'erez}
\address{Departamento de Matem\'atica PUC-Rio, Marqu\^es de S\~ao Vicente 225, G\'avea, Rio de Janeiro 225453-900, Brazil}
\email{lodiaz@mat.puc-rio.br}

\address{Centro de Matem\'atica da Universidade do Porto, Rua do Campo Alegre, 687, 4169-007 Porto, Portugal}
\email{sebastian.opazo@fc.up.pt}

\begin{document}

\begin{abstract} In dimension three and under certain regularity assumptions, we construct a renormalisation scheme at the heterodimensional tangency of a non-transverse heterodimensional cycle associated with a pair of saddle-foci whose  limit dynamic is a center-unstable H\'enon-like family displaying blender-horseshoes. As a consequence, the initial cycle can be approximated in higher regularity topologies by diffeomorphisms having blender-horseshoes.
\end{abstract}

\thanks{This paper is part of the PhD thesis of SP (PUC-Rio) supported by CNPq (Brazil).
The authors  thank the hospitality and support of Centro de Matem\'atica of Univ. of Porto (Portugal).
LJD is partially supported by CNE-Faperj and CNPq (Brazil) and SP is supported by CMUP (UID/MAT/00144/2013) and PTDC/MAT-CAL/3884/2014, which are funded by FCT (Portugal) with national (MEC) and European structural funds through the programs  COMPTE and FEDER, under the partnership agreement PT2020.}

\subjclass[2000]{Primary: 
}

\maketitle

\section{Introduction}
\label{s.intro1}
Homoclinic tangencies and
heterodimensional cycles
are the two main sources 
of non-hyperbolic chaotic dynamics. 
Palis' density  conjecture claims the these two types of bifurcations are the 
only obstructions to hyperbolicity: any non-hyperbolic system can be approximated by diffeomorphisms displaying 
some of these configurations, see \cite{Pal:00}. In the terminology of \cite{BonDiaVia:95} (see Preface), homoclinic tangencies
are in the core of the so-called {\emph{critical dynamics}} while heterodimensional cycles occur in 
{\emph{non-critical}} settings. Here {\emph{heterodimensional}} refers to the fact that the cycle involves saddles
whose unstable manifolds have different dimensions.
A crucial point is that critical and non-critical dynamics may occur simultaneously 
in some bifurcation scenarios and their effects may superimpose. This is precisely the setting of this paper.
We consider three-dimensional diffeomorphisms having simultaneously a heterodimensional cycle and a heterodimensional tangency (see part (b)  of Figure \ref{fig:homoandhete})
and study a renormalisation scheme leading to blender-horseshoes.
Our results continues the line of research started in  \cite{DiaKirShi:14}, where renormalisation schemes were introduced in 
this heterodimensional critical setting. The main difference with  \cite{DiaKirShi:14} is that we consider here cycles associated to a pair of saddle-foci instead  a pair of saddles with 
real multipliers.
We note that the dynamical setting of this paper is in some sense  a ``heterodimensional version'' of the 
{\emph{equidimensional}} (meaning that all  the saddles in the bifurcation have unstable manifolds of the same dimension)
configuration studied in \cite{GonShiTur08-2}, see part (a)  of Figure \ref{fig:homoandhete}. We now proceed to discuss more precisely the setting and the concepts involved in this paper.

\begin{figure}
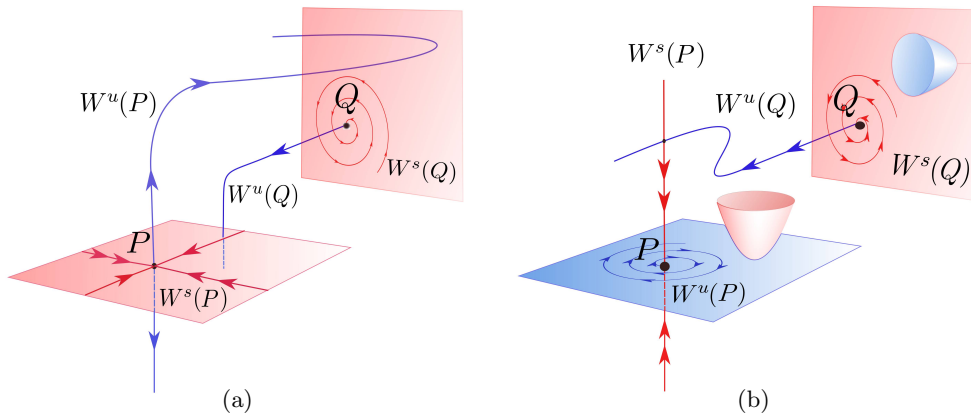

\centering
\begin{overpic}[scale=.065,
  ]{homovshete.jpg}
           \put(85,0){\small{(a)}}
          \put(280,0){\small{(b)}}
    \end{overpic}
\caption{(a) The non-transverse equidimensional cycle in \cite{GonShiTur08-2}. (b) A non-transverse heterodimensional cycle.}
\label{fig:homoandhete}
\end{figure}

\subsection*{Homoclinic bifurcations and renormalisation}
Homoclinic bifurcations may occur only in dimension two or higher.
Some important dynamical phenomena associated to homoclinic tangencies are the existence of horseshoes
whose invariant sets exhibit persistent non-transverse homoclinic
intersections, \cite{New:70,New:79}, the generic coexistence of infinitely many sinks/sources, \cite{New:74},
and the occurrence of H\'enon-like attractors/repellers, \cite{BenCar:91,MorVia:93}. Initially, these phenomena were studied
independently, but the study in \cite{PalTak:95} allows a unifying approach via renormalisation methods, see \cite[Chapter 3.4]{PalTak:95}. Roughly,  the quadratic family can be viewed  as a limit dynamics after a re-scaling process at a homoclinic tangency point. This approach  enables to translate persistent dynamical behaviours present in the quadratic family to the setting of homoclinic bifurcations. 
This is an important outcome of the renormalisation approach that can be used in another bifurcating settings
that is explored in this paper. We emphasise that the renormalisation methods have two parts: a (limit) family exhibiting some 
persistent dynamical ``interesting'' features and a sequence of ``dynamically defined''  maps 
approaching the initial configuration  whose ``returns''  (obtained by composing the diffeomorphisms) 
tend to this family. In this way, properties of the limit family are translated to the considered dynamical systems.
Let us postpone for a while the discussion about renormalisation and analyse  heterodimensional cycles.

\subsection*{Heterodimensional cycles}
First note that heterodimensional cycles can occur only on manifolds of dimension
three or higher. 
To keep our presentation as simple as possible, while describing all the essential complexity of these cycles, in what follows
we will assume that the dimension of the ambient manifold is three. 
We recall that
the {\emph{index}} of a hyperbolic periodic point $P$ of a diffeomorphism $f$ is the number of expanding eigenvalues 
of $Df^\pi (P)$, where $\pi$ is the period or $P$.
We say that a diffeomorphism $f$ has a {\emph{heterodimensional cycle}} associated to two saddles of different indices (here, by dimension constrains, the indices are necessarily two and one, or vice-versa) 
$P$ and $Q$ if the their invariant manifolds meet cyclically. In what follows, again for simplicity of the presentation, 
we will assume that $P$ and $Q$ are both fixed points and that $P$ has index two and $Q$ has index one.
Key aspects in the analysis of these cycles is the shape of
the intersections between these invariant sets as well as the type of expanding eigenvalues of $Df(P)$ and contracting eigenvalues
of $Df(Q)$. First,  due to dimension deficiency the one-dimensional invariant sets 
$W^{\mathrm s}(P)$ and $W^{\mathrm u}(Q)$ cannot have transverse intersections. Second, due to dimension sufficiency the two-dimensional invariant manifolds  $W^{\mathrm u}(P)$ and $W^{\mathrm s}(P)$ may have transverse and non-transverse intersections 
(and this two types of intersections may occur simultaneously). 
The combination of these aspects (types of eigenvalues of 
the saddles and types and shapes of the intersection of the invariant manifolds)
leads to a plethora of dynamical configurations and features (going from hyperbolic to a completely non-dominated dynamics)
that we do not aim to discuss in their full complexity\footnote{An interesting illustration of the complexity of these configurations can be done following the lines in the discussion in \cite[Section 2]{PalTak:87}.
Here the complexity is much higher:  the shapes of the intersections of the two-dimensional manifolds need to be considered.}.
 Instead, we prefer to outline two ``antipodal'' cases leading to two very different dynamical settings: partial hyperbolicity and
non-dominated dynamics.
Before presenting
these scenarios let us observe that we follow the approach proposed in \cite[Section 3]{PalTak:85}, we fix a 
{\emph{neighbourhood of the cycle}}
(roughly, an small open set containing the saddles of the cycle and intersections of their invariant manifolds)
and study the relative dynamics in such a set.
In what follows, we assume that the intersection of the one-dimensional invariant sets 
$W^{\mathrm s}(P)$ and $W^{\mathrm u}(Q)$ is \emph{quasi-transverse} (i.e., their tangent vectors at the intersection are not colinear).

\subsubsection*{Partial hyperbolicity versus non-domination}
A first (and certainly the simplest)
configuration occurs when all the eigenvalues of $P$ and $Q$ are real and different in modulus and the intersection
 between the two-dimensional invariant sets 
$W^{\mathrm u}(P)$ and $W^{\mathrm s}(Q)$ is transverse and defines a heteroclinic curve $\gamma$ with endpoints $P$ and $Q$.
The property of the eigenvalues allows to define the one-dimensional strong unstable foliation of
$W^{\mathrm u}(P)$ and 
 the one-dimensional strong unstable foliation of
$W^{\mathrm u}(Q)$. The ``simplest'' case occurs when $\gamma$ is transverse to these two foliations. This dynamical configuration is 
depicted in part (a) in Figure~\ref{fig.hciclo1}. 

\begin{figure}
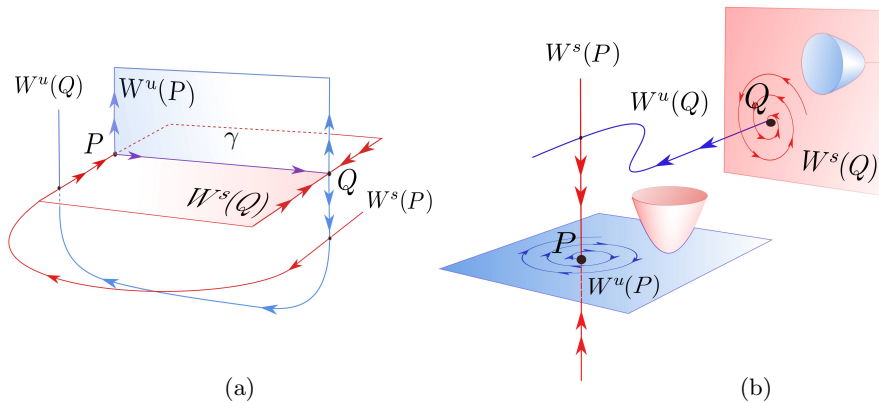

\centering
\begin{overpic}[scale=.063,
  ]{Cycle.jpg}
   \put(85,0){\small{(a)}}
          \put(280,0){\small{(b)}}
            \put(85,95){{$\gamma$}}
    \end{overpic}
\caption{Partial hyperbolic and non-dominated heterodimensional cycles.}
\label{fig.hciclo1}
\end{figure}

This type of configuration leads to partially hyperbolic dynamics with a one-dimensional center:  in the neighbourhood of the cycle there is a $Df$-splitting with three one-dimensional invariant directions
$E^{\mathrm{s}}\oplus E^{\mathrm{c}} \oplus E^{\mathrm{u}}$, $E^{\mathrm{s}}$ uniformly contracting, $E^{\mathrm{u}}$ uniformly expanding, and $E^{\mathrm{c}}$ has contracting and expanding regions.  For instance, this bifurcation can generate a pair of disjoint  hyperbolic sets (horseshoes) each one containing one of the saddles in the cycle \cite{DiaRoc:97} or a transitive set containing both saddles in the cycle and/or the phenomenon of intermingled homoclinic classes\footnote{A {\emph{homoclinic class}} of a saddle is the closure of the set of its transverse homoclinic points. In this particular case, the homoclinic classes of $P$ and $Q$ coincide, since these saddles
have different indices one has 
a non-hyperbolic transitive set.}, see
\cite{Dia:95a,Dia:95b,DiaRoc:01}, but this list does not exhaust all the possibilities of this configuration, see \cite{DiaRoc:02}. 

A second configuration occurs when either (i) the two-dimensional manifolds
$W^{\mathrm u}(P)$ and $W^{\mathrm s}(Q)$  have non-transverse intersections or (ii)
the saddles in the cycle have both a pair of non-real eigenvalues ($Df^\pi(P)$ a pair of non-real expanding eigenvalues
and $Df^\pi (Q)$ a pair of non-real contracting eigenvalues), see part (b) in Figure~\ref{fig.hciclo1}. These configurations lead to non-dominated dynamics:
there is no (non-trivial) bundle defined on the neighbourhood of the cycle that is $Df$-invariant. The second type of dynamics
was studied in \cite{BonDia:99}  (to get infinitely many sinks/sources) and in \cite{BonDia:03} (to get universal dynamics),
and the first type in \cite{DiaNogPuj:06} (where {\emph{heterodimensional tangencies}} were introduced). 
In this paper, we consider heterodimensional tangencies and heterodimensional cycles associated to saddles 
with non-real eigenvalues, continuing the study started in 
\cite{DiaKirShi:14}, where a similar configuration (with saddles having real eigenvalues) was considered.

Let us emphasise that in our setting 
the dynamics in the cycle is non-dominated, thus {\emph{finding relevant hyperbolic sets and seeing how these
sets are embedded in the global dynamics}} is a complicated task. 
Let us now make this sentence a bit more precise. In the case of homoclinic bifurcations, relevant sets are the so-called
{\emph{thick horseshoes}} which are the main responsible of the persistent non-hyperbolic features of these bifurcations
(for an ample discussion see \cite[Chapter 5]{PalTak:95}). In the case of heterodimensional cycles, it is proved in
\cite{BonDia:08,BonDia-HT:12} that these cycles yield {\emph{blenders}} (or more precisely, {\emph{blender-horseshoes}})
and that these blenders are in the core of most of the non-hyperbolic features related to these cycles (as for instance, intermingled
homoclinic classes and occurrence of robust cycles and tangencies). 
Our result state that, in our setting,  the  bifurcating diffeomorphisms $f$ can be approximated by  diffeomorphisms exhibiting blender-horseshoes, see Theorem~\ref{t.teo1} and Corollary~\ref{c.c1}. This approximation is done using a renormalisation scheme converging to a center-unstable H\'enon-like family 
(see equations \eqref{e.G} and \eqref{e.E})
and using the fact that this family exhibits blender-horseshoes, see \cite{DiaPer:18} and \cite{DiaKirShi:14} for a preliminary version.
In this approach, the H\'enon-like family plays the role of the quadratic family in the homoclinic setting.

\subsection*{Blenders and Blender-horseshoes}
Although blenders do not appear explicitly in our paper, they are a
fundamental object behind our constructions: the center-unstable H\'enon-like family provides these blenders as a side
effect. 
Let us say a few words about them.
A blenders is just a transitive  hyperbolic 
set  that occurs in dimension three or higher and
whose (local) stable set geometrically behaves as a set of  dimension 
larger than the one of its stable bundle. 
For an informal introduction to blenders we refer to
~\cite{BonCroDiaWil:16}, for 
some discussions on the notion and properties of blenders see \cite[Chapter 6.2]{BonDiaVia:06}.
Blender-horseshoes are a special kind of blenders, they are locally maximal hyperbolic sets 
conjugate to the complete shift in two symbols, 
for the precise definition of a blender-horseshoe see \cite{BonDia-HT:12}.
Comparing with the blenders in \cite{BonDia:96}, an important advantage of blender-horseshoes is that they can be
used to get robust cycles and tangencies, as in  \cite{BonDia:08,BonDia-HT:12}.
An important property of blender-horseshoes, that  we will use here in Corollary~\ref{c.c1}, is their $C^r$-persistence, see
\cite[Lemma 3.9]{BonDia-HT:12}.

\subsection*{Regularity}
So far, we have deliberately omitted (for simplicity of the discussion) 
any reference to the regularity of the perturbations of the systems. This will be a key point in the following discussion.
To put the statements as simple as possible (and slightly oversimplifying), the occurrence of thick horseshoes
in homoclinic bifurcations
 demands $C^2$-regularity
while the occurrence of blenders at heterodimensional cycles is obtained for $C^1$-perturbations. This means that
(partially) the two theories
are developed considering different degrees of differentiability.
Here it is important to recall the recent paper \cite{BarRai:17} where robust homoclinic tangencies
are obtained in some specific partially hyperbolic settings\footnote{This is a higher-dimensional result where the ambient space is four-dimensional:
to get the tangencies one needs a ``central direction'' of dimension at least two (besides some stable and unstable directions).
It is also important to recall that in dimension two there are no $C^1$-robust homoclinic tangencies, see \cite{Mor:11}.}. 
A natural goal is to get the occurrence
of robust heterodimensional cycles and homoclinic tangencies in $C^2$-settings.
The rough idea (that we will explore) is that one can go from $C^1$-regularity to higher regularity if the 
{\emph{dynamics of the intersection
is rich enough}} and {\emph{there are plenty of heteroclinic intersections}} between the saddles of the cycle.
It seems that the sort of bifurcations that we consider here provide such a  rich dynamics (this is an ongoing project,
see \cite{Perez:16}). 
Here two problems 
arise: the generation of blenders and their insertion  in the global dynamics.
One aims to ``relate"  these blenders to the two saddles in the initial cycle, but this is not always possible,
see \cite{BonDia-FC:12}.
A part of this problem (concerning the generation of blenders) 
was considered and solved in \cite{DiaKirShi:14}, but the embedding of these blenders in the dynamics
in \cite{DiaKirShi:14} is still not completely adequate.
The results in \cite{DiaKirShi:14} have three parts, it is proved that: (i) There is a center-unstable H\'enon-like family
displaying (for appropriate range of parameters) blender-horseshoes (see also \cite{DiaPer:18});
(ii) There is renormalisation scheme converging to that H\'enon-like family;
(iii) Some heteroclinic-like relations between the blender and the saddles in the cycle are obtained for appropriate $C^{1+\alpha}$, $\alpha<1$, perturbations. 
Here we see how (i) and (ii) hold in our setting and in a forthcoming paper we deal with the problem of the
dynamical embedding of the obtained blender and the heteroclinic relations, see also \cite{Perez:16}.

\subsection*{Final comments} We now discuss briefly some results in the literature related to our setting.
As observed, there are interesting scenarios where critical and non-critical dynamics are intermingled.
In the equidimensional
 context
let us recall the 
homoclinic and heteroclinic
two-dimensional non-transverse cycles
in \cite{TedYor86,GonShiSte:02, GonVlaSteVyaShi:06}
and,
in dimension three,
the series of papers \cite{Tat:01,GonGonTat:01,GonGonTat:07,GonMeiOvs06,GonShiTur08-2}, where 
homoclinic and heteroclinic non-transverse cycles were studied.
These papers involve renormalisation-like schemes
 leading to   families of H\'enon-like  maps or some variations:
  the logistic map in \cite{TedYor86}, generalisations of the H\'enon map in \cite{GonShiSte:02},
  Mira maps in \cite{GonGonTat:01,GonGonTat:07,Tat:01}, and three-dimensional H\'enon maps in \cite{GonOvsSimTur:05}.  
  See also the discussion in \cite[page 226]{GonShiTur08}.
  We observe that
 there are families of three-dimensional 
 H\'enon-like maps that exhibit 
new types of strange attractors, as for example the so-called \textit{wild-hyperbolic attractors}\footnote{Roughly, a wild-hyperbolic attractors possess the following two
distinctive properties: (1) wild-hyperbolic attractors allow homoclinic tangencies; (2) every such attractor and all nearby attractors (in the $C^r$-topology with $r \ge 2$) have no stable periodic orbits.}, for details see
\cite{ShiTur98, GonMeiOvs:06, GonOvsSimTur:05}.

On the other hand, 
 renormalisation methods do not seem to be sufficiently exploited in the heterodimensional context.
As far as we know, the first example in this direction was given in \cite{DiaKirShi:14}. Here we provide a renormalisations scheme for a non-transverse cycle  associated to pair of saddle-foci
points in the spirit of  \cite{DiaKirShi:14}.
Finally, since the configurations in this paper can be viewed as a ``heterodimensional version'' of  \cite{GonShiTur08-2}, a natural question is if they lead to families of  H\'enon-like maps with wild-hyperbolic attractors.

 Finally, let us recall that in the 
 heterodimensional setting, critical bifurcations were explored in \cite{KirSom:12},
  where heterodimensional tangencies and non-transverse cycles 
lead to $C^2$-robust heterodimensional tangencies (again, this is related to non-dominated dynamics).
In \cite{KirNisSom:10} it is shown how these tangencies lead to strange attractors.


\section{Statements of results} 
\label{s.introduction}
\subsection{Center-unstable H\'enon-like families}\label{ss.centerunstableblender}
We consider the 
\textit{center-unstable H\'e\-non-like} 
family of endomorphisms $G_{\xi,\mu,\kappa_1,\kappa_2}:\mathbb{R}^3\to \mathbb{R}^3$ defined by
\begin{equation}\label{e.G}
G_{\xi,\mu,\kappa_1,\kappa_2}(x,y,z) \eqdef \big(y,\mu +y^2+\kappa_1\,z^2+\kappa_2\,yz,\xi z+y\big).
\end{equation}

By \cite[Theorem 1]{DiaPer:18}, 
the family $G_{\xi,\mu,\kappa,\eta}$ has a blender-horseshoe
 for all
$(\xi,\mu,\kappa,\eta)$ in the open set 
$$
\mathcal{O}=  \mathcal{O}_\varepsilon \eqdef (1.18,1.19)\times (-10,-9)\times(-\varepsilon,\varepsilon)^2,
$$ 
for some $\varepsilon>0$.
See  \cite{DiaKirShi:14}  for  a version of this result  for blenders (instead of blender-horseshoes) and  \cite{HKOS:2018}  
for a complete numerical analysis of this family including the study of the creation and annihilation  
of blenders.

Our results also involve the endomorphisms family $E_{\xi,\mu,\bar \varsigma}:\mathbb{R}^3\to \mathbb{R}^3$ defined by
\begin{equation}\label{e.E}
E_{\xi,\mu,\bar \varsigma}(x,y,z)\eqdef
(\xi\,x +\varsigma_1 \,y,\, \mu + \varsigma_2\,y^2 + 
\varsigma_3\,x^2 + \varsigma_4\,x\,y,\, \varsigma_5\,y), 
\end{equation}
where $\bar \varsigma \eqdef (\varsigma_1,\varsigma_2,
\varsigma_3,\varsigma_4,\varsigma_5)$. 

Observe that if $\varsigma_1\varsigma_2\varsigma_5\neq 0$  then 
the families of endomorphisms 
$$
\big(\mu,E_{\xi,\mu,\bar \varsigma}\big)\quad \mbox{and}\quad\big(\mu,G_{\xi,\mu,\kappa,\eta}\big),\quad\mbox{where}\quad
\kappa=\varsigma_1^2\varsigma_3\varsigma_2^{-1},\quad \eta=\varsigma_1\varsigma_4\varsigma_2^{-1},
$$
are conjugate by the change of coordinates
\begin{equation*}
 \Theta:\mathbb{R}^4\to\mathbb{R}^4,\quad 
\Theta(\mu,x,y,z)=
(\varsigma_2^{-1}\mu,
\varsigma_2^{-1}\varsigma_1z,
\varsigma_2^{-1}y,
\varsigma_2^{-1}\varsigma_5x).
\end{equation*}

\subsection{Bifurcation setting and main result}
In what follows, $M$ denotes a closed boundaryless three manifold and $\diff^r(M)$, $r\ge 1$, the space of $C^r$-diffeomorphisms of $M$ endowed with the norm $\Vert \cdot \Vert_{C^r}$ of the uniform $C^r$-convergence.

We consider bifurcations of  diffeomorphisms $f\in \diff^r(M)$ having simultaneously a {\emph{heterodimensional cycle}}
and a {\emph{heterodimensional tangency.}} We assume that
$f$ has
a pair of
(periodic) saddles $P$ and $Q$ of {\emph{indices}} (dimension of the unstable bundle) two and one, respectively, such that  
conditions \textbf{(A)}-\textbf{(C)} below it holds:

\smallskip

\noindent
\textbf{(A)} {\emph{Saddle-focus periodic points}}: Let $\pi(P)$ and $\pi(Q)$ be  the periods of $P$ and $Q$.
Then
$Df^{\pi (P)}(P)$ has a pair of \textit{non-real expanding eigenvalues} and $Df^{\pi (Q)}(Q)$ has a pair of \textit{non-real contracting eigenvalues}. 
We assume that 
the restrictions of $f^{\pi (P)}$ 
and $f^{\pi (Q)}$
in small neighbourhoods 
of $P$ and $Q$ are $C^r$-linearisable. Some open and non-degenerate relations involving the eigenvalues of $Df^{\pi (P)}(P)$ and $Df^{\pi (Q)}(Q)$ are assumed, see \eqref{e.espectralconditions}.

\smallskip

\noindent
\textbf{(B)} \emph{Quasi-transverse intersection}: The one-dimensional invariant manifolds of $P$ and $Q$ intersect \textit{quasi-transversely} along the orbit of a
heteroclinic point $X$,  i.e., $X\in W^{\mathrm s}(P,f)\cap W^{\mathrm u}(Q,f)$ and
$$
T_X W^{\mathrm s}(P,f)+T_X W^{\mathrm u}(Q,f) =T_X W^{\mathrm s}(P,f)\oplus T_X W^{\mathrm u}(Q,f).
$$
Associated to the heteroclinic point $X$ there is 
a  \textit{transition map} corresponding to some iterate of $f$, going from a neighbourhood $U_Q$ of $Q$  to a neighbourhood  $U_P$ of $P$ 
following the orbit of $X$. Some conditions on this transition are given in
 \eqref{e.transition1}. 

\smallskip

\noindent
\textbf{(C)} \emph{Heterodimensional tangency}: 
The two-dimensional invariant manifolds of $P$ and $Q$ intersect along the orbit of a heteroclinic point $Y$ that is a
\textit{heterodimensional tangency}, i.e.,  the orbit of $Y$ is contained in the set 
$$
\big(W^{\mathrm u}(P,f)\cap W^{\mathrm s}(Q,f)\big)\setminus \big(W^{\mathrm u}(P,f)\pitchfork W^{\mathrm s}(Q,f)\big).
$$
Associated to the heteroclinic point $Y$ there is 
a  \textit{transition map} corresponding to some iterate of $f$, going from a neighbourhood $U_P$ of $P$  to a neighbourhood  $U_Q$ of $Q$ 
 following the orbit of $Y$. Some conditions on this transition are given in \eqref{e.transition2}. 

\smallskip

We are ready to state our main result.  

\begin{theorem}\label{t.teo1} 
Let $f:M\to M$ be a $C^r$-diffeomorphism, $r\ge 2$,   having a cycle associate to a pair of saddle-foci periodic points $P$ and $Q$ satisfying conditions 
 $\mathbf{(A)}$-$\mathbf{(C)}$. Then
there is a unfolding family 
$\mathfrak{F}=\{ f_{\bar \upsilon}\}_{ \bar \upsilon ̄\in \mathbb{R}^8}$ in $\diff^r(M)$ 
with  $f_{\bar 0}=f$ satisfying the following properties:

For every $\xi>0$ there exists a renormalisation scheme $\mathcal{R}(\xi,\mathfrak{F},f)$   
consisting of 
\begin{itemize}
\item[$\bullet$]
two sequences of natural numbers $(m_k)$ and $(n_k)$ with $(m_k), (n_k) \to +\infty$ as $k\to +\infty$;
\item[$\bullet$]
a sequence of $C^r$-parameterisations $\Psi_{k}:\mathbb{R}^3\rightarrow M$; 
\item[$\bullet$]
a sequence of $C^\infty$-functions $\bar{\upsilon}_{k}: \mathbb{R}\times [-\pi,\pi]^2\rightarrow  \mathbb{R}^8$
parameterising the bifurcation  parameter of the family  $\{ f_{\bar{\upsilon}} \}_{ \bar \upsilon ̄\in \mathbb{R}^8}$;
\item[$\bullet$]
a sequence of 
rescaled diffeomorphisms  $\mathcal{R}_{k}\big(f_{\bar{\upsilon}_{k}}\big): M\rightarrow M$, defined by 
$$
\mathcal{R}_{k}\big(f_{\bar{\upsilon}_{k}}\big)\eqdef
\big(f_{\bar{\upsilon}_{k}}\big)^{N_2+m_k+N_1+n_k},
$$
where 
 $N_1$ and $N_2$ are natural numbers independent of $k$ 
 and $\xi$, 
 \end{itemize}
satisfying the following convergence properties:
\begin{itemize}
\item[$\bullet$]
for every pair of compact sets
$\Delta\subset\mathbb{R}^3$ and 
$L\subset \mathbb{R}\times[-\pi,\pi]^2$ 
 it holds
$$
\Psi_{k}(\Delta)\rightarrow \{Y\},
\quad 
{\bar\upsilon}_{k}(L)\rightarrow \{ \textbf{\em{0}} \}
\quad
\mbox{when}
\quad 
k\rightarrow +\infty,
$$ 
where $Y$ is the heterodimensional tangency point of $f$ in Condition $\mathbf{(C)}$;
\item[$\bullet$]
if $\varphi_R$ is the argument  of the non-real eigenvalue of $Df^{\pi (R)} (R)$,  $R=P,Q$, then the 
renormalised sequence
\begin{equation*}
\Psi^{-1}_{k}\circ 
\mathcal{R}_{k}\big(f_{\bar{\upsilon}_{k}}\big)
\circ \Psi_{k}, \quad
\mbox{where} \quad \bar{\upsilon}_{k}={\bar\upsilon}_{k}(\mu,\varphi_P,\varphi_Q),\quad \mu\in\mathbb{R},
\end{equation*} 
converges
in the $C^r$-topology and on compact sets of $\mathbb{R}^3$
to the endomorphism
$E_{\xi,\mu,\bar \varsigma}$ in \eqref{e.E}, where $\bar \varsigma$
depends smoothly on $f$ and $\xi$. 
\end{itemize}
\end{theorem}

Equation  \eqref{e.2.24} provides the explicit formula for $\bar \varsigma= \bar \varsigma (f,\xi)$. 

The family of endomorphisms $E_{\xi,\mu,\bar \varsigma}$ is called the \textit{limit dynamics of the renormalisation scheme  $\mathcal{R}(\xi,\mathfrak{F},f)$}.

\begin{remark}\label{r.paramet}
{\em{The unfolding family in Theorem~\ref{t.teo1} involves eight parameters.
Three parameters for each 
non-transverse heteroclinic orbit and other two parameters to control the arguments of the saddle-foci.}}
\end{remark}

\begin{remark}\label{r.hayblenders}
{\em{An adequate choice of  
$\bar \varsigma=\bar \varsigma(f,\xi)$
guarantees that the diffeomorphisms $f_{\bar{\upsilon}_{k}}$ have blender-horseshoes  for every $k$ large enough. More precisely,
recalling the observations in Section~\ref{ss.centerunstableblender}, it follows that
if  $\bar \varsigma=(\varsigma_1,\varsigma_2,
\varsigma_3,\varsigma_4,\varsigma_5)$ is such that
 \begin{equation}\label{e.restrictions}
 \varsigma_1\varsigma_2\varsigma_5\neq 0\,\,\,\mbox{and}\,\,\,
\big(\xi,\mu,
\varsigma_1^2\varsigma_3\varsigma_2^{-1},\varsigma_1\varsigma_5\varsigma_2^{-1}
\big) 
 \in \mathcal{O},
 \end{equation}
 then the initial cycle is $C^r$-approximated by diffeomorphisms having blender-horse\-shoes arbitrarily close to the heterodimensional tangency of $f$.}}
\end{remark}

Let us restate Remark~\ref{r.hayblenders} using the $C^r$-persistence of blender-horseshoes
in \cite[Lemma 3.9]{BonDia-HT:12}.

\begin{cor}\label{c.c1}
Let $f:M\to M$ be a $C^r$ diffeomorphism as in Theorem \ref{t.teo1}. For each  $\xi\in (1.18,1.19)$ consider the renormalisation scheme 
 $\mathcal{R}(\xi,\mathfrak{F},f)$  of $f$ and its  limit dynamics 
$E_{\xi,\mu,\bar \varsigma}$.
If  $\bar \varsigma$ satisfies  equation \eqref{e.restrictions} then  $f_{\bar{\upsilon}_{k}}$ has a blender-horseshoe
nearby the  heterodimensional tangency of $f$ for every $k$ large enough. In particular, for every $k$ large enough, there is a
$C^r$-open set  
$\mathcal{U}_{k}$ of diffeomorphisms $C^r$-close to $f$ such that each $g\in \mathcal{U}_{k}$ has a blender-horseshoe nearby the  heterodimensional tangency of $f$. 
\end{cor}

We observe that there are diffeomorphisms $f$ satisfying conditions  \textbf{(A)}-\textbf{(C)} whose vector 
$\bar \varsigma=\bar\varsigma (f,\xi)$ satisfies the condition \eqref{e.restrictions} for every  $\xi \in (1.18,1.19)$.
This  point will be proved in Lemma~\ref{l.ultimolema} after introducing the appropriate terminology.

\subsection*{Organisation of  the paper} 
%
Conditions (\textbf{A})-(\textbf{C}) are precisely stated in Section~\ref{s.descriptionAC}.
In Section~\ref{s.newunfolding},
we describe the perturbations in the renormalisation scheme of Theorem \ref{t.teo1}.
The unfolding family $\mathfrak{F}$ is defined in Section~\ref{ss.unfolding}.
In Section~\ref{s.ProofofTheoremt.teo1} we analyse the dynamics of the returns
in the renormalisation scheme (i.e., compositions of 
the form $f^{N_2}_{\bar\upsilon}\circ f^{m}_{\bar\upsilon}\circ f^{N_1}_{\bar\upsilon}\circ f^{n}_{\bar\upsilon}$
defined nearby the heterodimensional tangency). 
 The convergence of the renormalised sequence
is obtained in  Section~\ref{ss.conclusion} after obtaining an explicit formula for these maps
in Section~\ref{s.newProofofTheoremt.teo1}. The existence of diffeomorphisms satisfying Corollary~\ref{c.c1} is proved in
Section \ref{ss.ultima}.

\section{Description of the cycle: Conditions (\textbf{A})-(\textbf{C})}
\label{s.descriptionAC}
In this section, we describe in precise form the conditions (\textbf{A})-(\textbf{C}) in Section~\ref{s.introduction}.

\subsection{(\textbf{A}) Local dynamics at $P$ and $Q$}
\label{ss.localdynamicsatsaddlepoints}
Without loss of generality, let us assume that $P$ and $Q$ are fixed points of $f$. 
We assume the existence of 
local $C^r$-linearising charts 
 $U_P, U_Q=(-10,10)^3$ at the saddles $P$ and $Q$ (here the corresponding saddle is identified with the origin) such that 
the expression of $f$ in these neighbourhoods is of the form

\begin{equation}
\label{e.localdynamics}
\begin{split}
&f|_{U_P}=\left( \begin{array}{ccc}
\lambda_P & 0 & 0 \\
0 & \sigma_P\cos(2\pi
\varphi_P) & -\sigma_P
\sin(2\pi \varphi_P)\\
0 & \sigma_P
\sin(2\pi \varphi_P) & \sigma_P
\cos(2\pi \varphi_P)\end{array} \right),\
\quad
\mbox{and}
 \\
&f|_{U_Q}=\left( \begin{array}{ccc}
\lambda_Q \cos(2\pi\varphi_Q) & 0 & -\lambda_Q \sin(2\pi\varphi_Q) \\
0 & \sigma_Q & 0\\
\lambda_Q \sin(2\pi\varphi_Q) & 0 & \lambda_Q \cos(2\pi\varphi_Q)\end{array} \right),
\end{split}
\end{equation}
where $\varphi_P, \varphi_Q\in [0,1]$, $\varphi_P\ne \varphi_Q$, and $\lambda_P, \lambda_Q, \sigma_P,\sigma_Q\in \mathbb{R}$  are such that
$$
0<\vert \lambda_P\vert,\,\vert \lambda_Q\vert<1<
\vert\sigma_P\vert,\,
\vert\sigma_Q\vert.
$$
We also assume the
following condition  called   \textit{spectral condition of the cycle,}
\begin{equation}\label{e.espectralconditions}
0<\left|
\big||\lambda_{P}|^{\frac1{2}}\,\sigma_{P}
\big|^{\eta}\sigma_Q \right|<1,\quad\mbox{where}\quad \eta=\dfrac{\log|\lambda_{Q}^{-1}|}{\log |\sigma_P|}.
\end{equation}
This condition
plays a key role in the convergence of
the renormalisation scheme.
Lemma~\ref{l.p.spectralconditions} claims that
 this spectral condition is open and non-empty. 
 
\subsubsection{Spectral conditions on the cycle}
\label{sss.rango}

%
Taking the square of $f$, if necessary,
 we can assume that $\lambda_P,\sigma_P,\lambda_Q$ and $\sigma_Q$ are all positive. With this in mind,
 in the next proposition we show that the condition~\eqref{e.seis} is non degenerate.
As consequence the set of diffeomorphisms $f$ such that its eigenvalues 
satisfy the spectral condition \eqref{e.seis} is an   open set in $\diff^r(M)$.



%
%

\begin{lemma}
\label{l.p.spectralconditions}
Let $\mathcal{P}$ be the set of points
 $(\tilde{\lambda},\tilde{\sigma},\lambda,\sigma)\in \mathbb{R}^4$ such that $0<\lambda,\tilde{\lambda}<1
 $ and $\tilde{\sigma},\sigma>1$ and 
\begin{equation}
\label{e.seis}
0<(\tilde{\lambda}^{\frac1{2}}\,\tilde{\sigma})^{\eta}\sigma <1,\quad\mbox{where}\quad \eta=\frac{\log\lambda^{-1}}{\log\tilde{\sigma}}.
\end{equation}
The set $\mathcal{P}$ is non-empty and open. 
\end{lemma}
\begin{proof}
The set $\mathcal{P}$ is open by definition. 
Thus it remains to show the existence of numbers satisfying these inequalities. 
For this, consider the set
\begin{equation}
\label{e.Zset}
\widetilde{\mathcal{Z}}\eqdef  \big\{
(\tilde{\lambda},\tilde{\sigma})\in (0,1)\times(1,+\infty): 0<\tilde{\lambda}^{\frac1{2}}\,\tilde{\sigma}<1 \big\}\subset \mathbb{R}^2.
\end{equation}
The lemma follows the next claim.
\begin{cl}
\label{c.l.limon}
For $(\tilde{\lambda},\tilde{\sigma},\lambda)\in \widetilde{\mathcal{Z}}\times (0,1)$
there is a interval 
$I_{(\lambda,\tilde{\lambda},\tilde{\sigma})}$,
such that every $(\tilde{\lambda},\tilde{\sigma},\lambda,\sigma)\in
\widetilde{\mathcal{Z}}\times (0,1)\times I_{(\lambda,\tilde{\lambda},\tilde{\sigma})}$
satisfies  \eqref{e.seis}. 
\end{cl}

\begin{proof}The equation~\eqref{e.seis}
is equivalent to 
$$
\dfrac{\log\lambda^{-1}}{\log\tilde{\sigma}}\,\log(\tilde{\lambda}^{\frac1{2}}\,
\tilde{\sigma})+\log\sigma<0.
$$
Note that every vector $(\tilde{\lambda},\tilde{\sigma},\lambda)$ in $\widetilde{\mathcal{Z}}\times (0,1)$
satisfies the inequality 
$$
\dfrac{\log\lambda^{-1}}{\log\tilde{\sigma}}\,\log(\tilde{\lambda}^{\frac1{2}}\,\tilde{\sigma})<0.
$$
Thus for every $\sigma>1$ such that
\begin{equation}
\label{e.ecu}
\frac{\log\lambda^{-1}}{\log\tilde{\sigma}}
\log(\tilde{\lambda}^{\frac1{2}}\,\tilde{\sigma})<-\log\sigma,
\end{equation}
we get that $(\tilde{\lambda},\tilde{\sigma},\lambda,\sigma)$ satisfies  \eqref{e.seis}.
Now it is enough to take  $I_{(\lambda,\tilde{\lambda},\tilde{\sigma})}=(1,\sigma^{*}_
{(\lambda,\tilde{\lambda},\tilde{\sigma})})$, where $\sigma^*_{(\lambda,\tilde{\lambda},\tilde{\sigma})}$ is the supremum of the $\sigma>1$ satisfying ~\eqref{e.ecu}.
\end{proof}
The proof of the lemma is now complete.
\end{proof}

\subsection{(\textbf{B}) Quasi-transverse intersection and transition from $Q$ to $P$.}\label{ss.CB}
We now describe the transition from $Q$ to $P$ along the  orbit of the quasi-transverse intersection point
$X\in W^\mathrm{s}(P,f)\cap W^\mathrm{u}(Q,f)$.

Using  Condition (\textbf{A}), we can consider the local invariant manifolds of $P$ and $Q$ defined by 
$$
W^\mathrm{s}_{\mathrm{loc}}(P,f)\eqdef(-10,10)\times\{(0,0)\},
\qquad
 W^\mathrm{u}_{\mathrm{loc}}(P,f)\eqdef\{0\}\times(-10,10)^2$$
 and 
 $$
W^\mathrm{s}_{\mathrm{loc}}(Q,f) \eqdef(-10,10)\times\{0\}\times(-10,10), 
\qquad
W^\mathrm{u}_{\mathrm{loc}}(Q,f)\eqdef\{0\}\times(-10,10)\times\{0\}.
$$

Replacing $X$ by some backward iterate, if necessary, we can assume that $X\in W^\mathrm{u}_{\mathrm{loc}}(Q,f)$. Rescalling the segment 
$W^\mathrm{u}_{\mathrm{loc}}(Q,f)$, if necessary, we can assume  
 that $X = (0, 1, 0)\in U_Q$. By hypothesis, there exists $N_1\in \mathbb{N}$ such that 
$\widetilde{X}=f^{N_1}(X)\in W^\mathrm{s}_{\mathrm{loc}}(P,f)\subset U_P$.
We choose $N_1$ so that  $f^{N_1-1}({X})\notin U_P$ (i.e. $\widetilde{X}$ is the first iterated of the orbit of $X$ meeting $U_P$). Arguing as before, we can take $\widetilde{X} = (1, 0, 0)\in U_P$.
We assume that there are
small neighbourhoods $U_X$ of $X$ in $U_Q$ and  $U_{\widetilde{X}}$
of $\widetilde{X}$ in $U_P$ such that the map
$f^{N_1}|_{U_X} \colon U_X \to U_{\widetilde{X}}$
is given by
\begin{eqnarray}\label{e.transition1}
f^{N_1}\left(\begin{array}{cc} x \\ y+1 \\ z \end{array} \right)=
\left( \begin{array}{ccc}
1+\alpha_1 x + \alpha_2 y +\alpha_3 z +\widetilde{H}_1(x,y,z) \\
 \beta_1 x+\beta_2 y+\beta_3 z+ \widetilde{H}_2(x,y,z)\\
\gamma_1 x+\gamma_2 y+\gamma_3 z + \widetilde{H}_3(x,y,z)\end{array} \right),
\end{eqnarray}
where $\alpha_i,\beta_i,\gamma_i$, $i=1,2,3$ are constants such that
\begin{equation}
\label{e.ddd}
\beta_1=\beta_3=\gamma_1=\gamma_2=0
\end{equation}
and  the maps $\widetilde{H}_i$, $i=1,2,3$,  are higher order terms 
 satisfying
\begin{equation}
\label{e.H_i}
\widetilde{H}_i(\textbf{0})=\frac{\partial}{\partial x}\widetilde{H}_i(\textbf{0}) = \frac{\partial}{\partial y}\widetilde{H}_i(\textbf{0}) = \frac{\partial}{\partial z}\widetilde{H}_i(\textbf{0})=0.
\end{equation}
Note that as $f^{N_1}$ is a (local) diffeomorphism it holds that 
\begin{equation}
\label{e.dd}
\alpha_1\, \beta_2\,\gamma_3\ne 0.
\end{equation}
We call $f^{N_1}$ and $N_1$
\textit{transition map} and  \emph{transition time from $Q$ to $P$}, respectively.

\begin{remark}\label{r.homotanggencypointextremis}{\em{
The vector $\bar u=(0,1,0)$ spans $T_X W^{\mathrm{u}}(Q,f)$ and $Df^{N_1}(X) (\bar u) =(\alpha_2,\beta_2,0)$,
Since, by \eqref{e.dd}, $\beta_2\ne 0$, this vector does not belong to
to $T_{\tilde X} W^{\mathrm{s}}(P,f)$ (which is spanned by $(1,0,0)$). Thus $X$ is a quasi-transverse heteroclinic point.
}}
\end{remark}

\subsection{(\textbf{C}) Heterodimensional tangency and transition from $P$ to $Q$}\label{ss.CC}
We now describe the transition from $P$ to $Q$ along the  orbit of the heterodimensional tangency point 
$Y\in W^\mathrm{u}(P,f)\cap W^\mathrm{s}(Q,f)$.
 By replacing $Y$ by some backward iterate we can assume that $Y\in W^\mathrm{u}_{\mathrm{loc}}(P,f)$.
 By hypothesis, there is $N_2\in \mathbb{N}$ 
 such that $\widetilde{Y} = f^{N_2}(Y)\in W^\mathrm{s}_{\mathrm{loc}}(Q,f)$. We choose $N_2$ so that  $f^{N_2-1}({Y})\notin U_Q$.
By some linear coordinate change in $U_P$ and in $U_Q$, one can assume  $Y = (0, 1, 1)\in U_P$ and 
$\widetilde{Y} = (1, 0, 1)\in U_Q$. Note that this coordinate change  can be done without changing the previous 
choice of  $X$ and $\widetilde{X}$.
We assume that there are
small neighbourhoods $U_Y$ of $Y$ in $U_P$ and  $U_{\widetilde{Y}}$
of $\widetilde{Y}$ in $U_Q$ such that the map
$f^{N_2}|_{U_Y} \colon U_Y \to U_{\widetilde{Y}}$
is given by
\begin{eqnarray}
\label{e.transition2}
\quad f^{N_2}\left(\begin{array}{cc} x \\ 1+y \\ 1+z \end{array} \right)=\ \left(\begin{array}{cc}1+ a_1 x + a_2y + a_3 z+H_1(x,y,z) \\ b_1 x + b_2 y^2+ b_3 z^2 +b_4 yz+H_2(x,y,z) \\1+ c_1 x + c_2 y+ c_3 z+H_3(x,y,z)\end{array} \right),
\end{eqnarray}
where $a_i,b_i,c_i$, $i=1,2,3$, are constants with
\begin{equation}
\label{e.bbs}
c_2=c_3
\end{equation}
and the maps $H_i$, $i=1,2,3$, are higher order terms  satisfying the following conditions:
\begin{equation}\label{e.hs}
\begin{split}
&
\quad 
H_i(\textbf{0})=\frac{\partial}{\partial x}H_i(\textbf{0}) = \frac{\partial}{\partial y}H_i(\textbf{0}) = \frac{\partial}{\partial z}H_i(\textbf{0})=0,
\\
&
\quad 
\frac{\partial^2}{\partial y^2}H_2(\textbf{0})= \frac{\partial^2}{\partial z^2}H_2(\textbf{0}) = \frac{\partial^2}{\partial y\,\partial z}H_2(\textbf{0})= 0.
\\
\end{split}
\end{equation}
Note that since $f^{N_2}$ is a (local) diffeomorphism it  follows that 
\begin{equation}
\label{e.d}
b_1\,c_2\,(a_3-a_2)\neq 0.
\end{equation}
We call $f^{N_2}$ and $N_2$
\textit{transition map} and  \emph{transition time from $P$ to $Q$}, respectively.

Finally, 
we assume the following condition on the parameters $a_2,a_3$, and $\gamma_3$ above
\begin{equation}
\label{e.>}
\gamma_3\,(a_3-a_2)>0.
\end{equation}  
Note that the choices of $a_2,a_3,$ and $\gamma$ are compatible with \eqref{e.d} and \eqref{e.dd}.

\begin{remark}\label{r.heteroclinicpointextremis}{\em{
The vectors $\bar v=(0,1,0)$ and $\bar w=(0,0,1)$ span $T_Y W^{\mathrm{u}}(P,f)$ and 
$Df^{N_2}(Y) (\bar v) =(a_2,0,c_2)$ and 
$Df^{N_2}(Y) (\bar w) =(a_3,0,c_3)$. Since these vectors belong to
 $T_{\widetilde Y} W^{\mathrm{s}}(Q,f)$ (which is spanned by $(1,0,0)$ and $(0,0,1)$), the point $Y$ corresponds to a heterodimensional tangency.
}}
\end{remark}

\section{Translation and rotation-like perturbations}\label{s.newunfolding}
In this section, we describe the two types of perturbations (translation and rotation like)
involved in the renormalisation scheme of Theorem \ref{t.teo1}.

We consider auxiliary $C^r$-bump functions $b_{\rho}:\mathbb{R}\to \mathbb{R}$, $\rho>0$,
satisfying 
\begin{equation*}
 \left\{
\begin{array}{c l}
b_{\rho}(x) = 0,  &  \mbox{if}\quad  |x|  \geq \rho ,\\
0 < b_{\rho}(x)< 1,  &  \mbox{if}\quad  \dfrac{\rho}{2}< |x|< \rho,\\
b_{\rho}(x) = 1, & \mbox{if}\quad |x| \leq \dfrac{\rho}{2},
\end{array}
\right.
\end{equation*}
and their associated  $C^r$-bump functions $B_{\rho}:\mathbb{R}^3\to \mathbb{R}$ defined by 
 \begin{equation*}
B_{\rho}(x, y, z) \eqdef b_{\rho}(x)\, b_{\rho}(y)\, b_{\rho}(z).
\end{equation*}

\subsection{Translation-like perturbations}
\label{ss.unfoldingper}
For $Z_0\in\mathbb{R}^3$ denote by $\mathbb{B}_\rho(Z_0)\subset \mathbb{R}^3$ the open ball of radius $\rho$ and center $Z_0$. 
Given $Z_0\in \mathbb{R}^3$, consider the family of $C^r$-maps
\begin{equation}\label{e.traslationfamily}
T_{Z_0,\bar w}: \mathbb{R}^3\to \mathbb{R}^3,\quad \bar w \in \mathbb{R}^3,
\end{equation}
defined by 
\begin{itemize}
 \item
 if $Z+Z_0 \in \bB_{\rho}(Z_0)$, then
$T_{Z_0,\bar w}(Z+Z_0 ) = Z+Z_0 + B_{\rho}(Z)\bar w$,
\item
 if $Z\not\in \bB_{\rho}(Z_0)$, then $T_{Z_0,\bar w}(Z )=Z$.
 \end{itemize}
Note that 
$$
\big\Vert
 T_{Z_0,\bar w}- \mathrm{id} \big\Vert_{C^r} \le \big\Vert B_{\rho} \big\Vert_{C^r}\cdot ||\bar w ||.
 $$
Thus, for every $||\bar w||$ small enough, 
the map $ T_{Z_0,\bar w}$ is a $C^r$-perturbation of the identity supported in $\bB_{\rho}(Z_0)$. Note also that, by construction,
$$
T_{Z_0,\bar w}(\bB_{\rho}(Z_0))= \bB_{\rho}(Z_0).
$$
We call $T_{Z_0,\bar w}$   a \textit{translation-like perturbation} of the identity.

\subsection{Rotation-like perturbations}
\label{ss.rotationperturbations1}
Consider the  families of linear maps
$$
I^x_{\omega},\,I^y_{\omega}\colon \mathbb{R}^3\to\mathbb{R}^3,\quad \omega \in [-\pi,\pi],
$$
given by 
\begin{eqnarray*}
I^x_{\omega}\eqdef \left( \begin{array}{ccc}
1 & 0 & 0 \\
0 & \cos(2\pi\omega) & -\sin(2\pi\omega) \\
0 & \sin(2\pi\omega) & \cos(2\pi\omega)\end{array} \right),
\quad
I^y_{\omega}\eqdef \left( \begin{array}{ccc}
\cos(2\pi\omega) & 0 & -\sin(2\pi\omega) \\
0 & 1 & 0\\
\sin(2\pi\omega) & 0 & \cos(2\pi\omega)\end{array} \right).
\end{eqnarray*}
Observe that if $\omega=0$ then $I^x_{0}=I^y_{0}=\mathrm{id}$.  

 Consider the families of $C^r$-diffeomorphisms 
\begin{equation}\label{e.rotationfamily}
\begin{split}
&R^x_{{\omega},\rho} \colon \mathbb{R}^3\to \mathbb{R}^3, \quad R^x_{{\omega,\rho}}(W)=
I^x_{\omega\, b_\rho(||W||)}\cdot
W^T,\\
&R^y_{{\omega,\rho}} \colon \mathbb{R}^3\to \mathbb{R}^3, \quad R^y_{{\omega,\rho}}(W)=
I^y_{\omega\, b_\rho(||W||)}\cdot
W^T
\end{split}
\end{equation}
where $W^T$ denotes the  transpose of the vector $W\in\mathbb{R}^3$.

Note that the restriction of $R^x_{{\omega,\rho}}$  to the set $[-\frac{\rho}{2},\frac{\rho}{2}]^3$ coincides with $I^x_{\omega}$ and  $R^x_{{\omega,\rho}}
$ is the identity map in the complement of $[-\rho,\rho]^3$. Analogously, 
$R^y_{{\omega,\rho}}=I^y_{\omega}$ in  $[-\frac{\rho}{2},\frac{\rho}{2}]^3$ and 
$R^y_{{\omega,\rho}}=\mathrm{id}$ 
in  $\mathbb{R}^3\setminus[-\rho,\rho]^3$. Moreover,
$$
R^x_{{\omega,\rho}}\big([-\rho,\rho]^3\big)= 
R^y_{{\omega,\rho}}\big([-\rho,\rho]^3\big)=
[-\rho,\rho]^3.
$$ 
%
%

It is not hard to see that there are  constants $C_\rho$ and $C'_\rho$ with
$$
\big\Vert
R^x_{{\omega,\rho}}- \mathrm{id} \big\Vert_{C^r} < C_\rho |{\omega}|
\quad \mbox{and} \quad
\big\Vert
R^y_{{\omega,\rho}}- \mathrm{id} \big\Vert_{C^r} < C'_\rho |{\omega}|.
 $$  
Thus, for every $\omega$ small enough, 
the maps $R^x_{{\omega,\rho}}$ and $R^y_{{\omega,\rho}}$ are $C^r$-perturbations of identity supported in $[-\rho,\rho]^3$.

We call $R^x_{{\omega,\rho}}$ and $R^y_{{\omega,\rho}}$  \textit{rotation-like perturbations} of the identity.

%

\section{The unfolding family $\mathfrak{F}$}
\label{ss.unfolding}

We now describe the $8$-parameter family
$\mathfrak{F}=\{f_{\bar\upsilon}\}_{\bar\upsilon}$ of $C^r$-diffeomorphisms, $r\ge 2$, unfolding the cycle of $f$ at $\bar\upsilon=\bar 0\in \mathbb{R}^8$ (i.e., $f_{\bar{0}} = f$) in Theorem~\ref{t.teo1}.
The cycle of $f$ has two parts with (say) ``independent'' unfoldings, one associated to the heterodimensional tangency and another one associated to the
quasi-transverse heteroclinic point.  
Besides the unfolding of these 
heteroclinic non-transverse orbits we need to consider
slight adjustments on the arguments of the eigenvalues of the saddles $P$ and $Q$. 
These adjustments are given by the rotation-like perturbations in Section \ref{s.newunfolding}.
In summary, the family $\mathfrak{F}$ is obtained considering  translation-like perturbations nearby  
the heteroclinic  points $\widetilde{X}$ and  $\widetilde{Y}$ and 
rotation-like perturbations corresponding to
small changes  of  the arguments $\varphi_P$ and $\varphi_Q$. 
More precisely, the unfolding family is of the form
\begin{equation}
 \label{e.thefamily}
f_{\bar\upsilon} = \Gamma_{\bar \upsilon}\circ f,
\quad \mbox{with}\quad \bar\upsilon=(\bar \mu,\bar \nu,{\alpha},\beta)\in [-\epsilon,\epsilon]^8,
\end{equation}
where $\Gamma_{\bar \upsilon}$ is the perturbation of the identity obtained as follows.
For $R=P,Q,\widetilde{X},\widetilde{Y}$
we take pairwise disjoint neighbourhoods $V_R\subset U_P\cup U_Q$ and let  
\begin{equation}
\label{e.definitionofGa}
\Gamma_{\bar\upsilon=(\bar \mu,\bar \nu,{\alpha},\beta)}(Z)=
\left\{
\begin{split}
&\Gamma^P_{\alpha}(Z)\eqdef R^x_{\alpha,\rho}(Z), \quad \mbox{if $Z\in V_P$}
,\\
&\Gamma^{\q}_{\bar \nu} (Z)\eqdef T_{\widetilde{X},\bar \nu}(Z), \quad \mbox{if $Z\in V_{\widetilde{X}}$}
,\\
& \Gamma^Q_{\beta}(Z)\eqdef R^y_{\beta,\rho}(Z), \quad \mbox{if $Z\in V_Q$},\\
&\Gamma^{\h}_{\bar \mu}(Z)\eqdef T_{\widetilde{Y},\bar \mu}(Z), \quad \mbox{if $Z\in V_{\widetilde{Y}}$}
,\\
& \mathrm{id}(Z), \quad \mbox{\,\,\,if $Z\not\in V_P\cup V_Q\cup V_X \cup V_Y$},
\end{split}
\right.
\end{equation}
where $\rho>0$ is small $T_{Z,\bar w}, R^x_{\omega,\rho}$ and $R^y_{\omega,\rho}$ are as in Section \ref{s.newunfolding}.
The precise choice of these neighbourhoods $V_R$ is done below
and the choice of $\rho$ in Sections~\ref{ss.unfoldingperturbation} and \ref{ss.rotationperturbations2}. Properties of the map $\Gamma_{\bar \upsilon}$ are discussed in
Section~\ref{ss.propertiesoftheunfoldingfamily}.
Finally,
we call 
$ \Gamma^\h_{\bar \mu}$ and $\Gamma^\q_{\bar \nu}$  \emph{unfolding perturbations}  and 
 $ \Gamma^P_{\alpha}$ and $\Gamma^Q_{\beta}$ \emph{rotating perturbations}.  

We recall in Conditions (\textbf{A})-(\textbf{C}), the  definition of the neighbourhoods
\begin{equation*}
U_R=(-10,10)^3\ni R\quad\mbox{and}\quad 
U_{\widetilde{Z}}\ni \widetilde{Z},  
\quad\mbox{where}\quad R=P,Q,\quad Z=X,Y. 
\end{equation*}
We recall that by the choice of
$N_1$ and $N_2$ in Conditions (\textbf{B})-(\textbf{C}),
 the points 
$\widetilde{X}$ and $\widetilde{Y}$ satisfies
$f^{-1}(\widetilde{X}) \notin U_P$ and $f^{-1}(\widetilde{Y})\notin U_Q$.

Consider the neighbourhoods 
$V_R=[-8,8]^3\subset U_R$ of $R=P,Q$.
\subsection{The unfolding perturbations $\Gamma^\q_{\bar \mu}$ and $\Gamma^\h_{\bar \nu}$}
\label{ss.unfoldingperturbation} 
 Consider a small enough $\rho>0$ such that 
\[
\begin{split}
 &\mathbb{B}_{\rho}(\widetilde{X})\subset U_{\widetilde{X}},\quad
f^{-1}\Big(\mathrm{closure}\big(\mathbb{B}_{\rho}(\widetilde{X})\big)\Big)\cap U_P= \emptyset,\quad\mbox{and}
\\
&\mathbb{B}_{\rho}(\widetilde{Y})\subset U_{\widetilde{Y}},\quad
f^{-1}\Big(\mathrm{closure}\big(\mathbb{B}_{\rho}(\widetilde{Y})\big)\Big)\cap U_Q= \emptyset.
\end{split}
\]
In particular,
$$
f(\mathbb{B}_{\rho}(\widetilde{Z}))\cap \mathbb{B}_{\rho}(\widetilde{Z}) = \emptyset,
\quad Z=X,Y;
$$
and
 $$
 P\notin \mathrm{closure}(\bB_{\rho}(\widetilde{X}))\quad\mbox{and}\quad Q\notin \mathrm{closure}({\bB}_{\rho}(\widetilde{Y})).
 $$
For $\bar \mu$ and $\bar \nu$  in $\mathbb{R}^3$
we define the \textit{unfolding perturbations}  as
\[
\Gamma^\h_{\bar \mu}\colon M\to M,\quad \Gamma^\h_{\bar \mu}\eqdef T_{\widetilde{Y},\bar{\mu}}
\quad\mbox{and}\quad 
\Gamma^\q_{\bar \nu}\colon M\to M,\quad \Gamma^\q_{\bar \nu}\eqdef T_{\widetilde{X},\bar{\nu}},
\]
where $T_{Z_0,\bar w}$ are the translation-like perturbations in \eqref{e.traslationfamily}.
See Figure~\ref{fig:Unfolding}.

\begin{figure}[htb]
 \centering
 \includegraphics[width=0.8\textwidth]{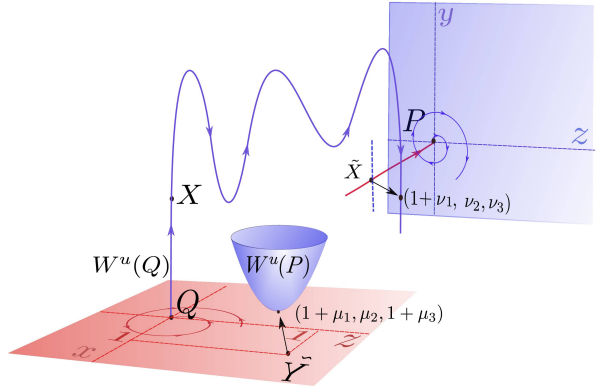}
 \caption{The unfolding family $\mathfrak{F}$.}
  \label{fig:Unfolding}
 \end{figure}


As was observed, if $||\bar \mu||$ and $||\bar \nu||$ are small enough the maps $\Gamma^\h_{\bar \mu}$ and $\Gamma^\q_{\bar \nu}$ are $C^r$-perturbations of the identity  
supported in $\bB_{\rho}(\widetilde{Y})$ and $\bB_{\rho}(\widetilde{X})$, respectively, 
with
$$
\Gamma^\q_{\bar \nu}(\bB_{\rho}(\widetilde{X}))= \bB_{\rho}(\widetilde{X})\quad \mbox{and}\quad 
\Gamma^\h_{\bar \mu}(\bB_{\rho}(\widetilde{Y}))=\bB_{\rho}(\widetilde{Y}).
$$

\subsection{The rotating  perturbations
$\Gamma^P_{{\alpha}}$ and $\Gamma^Q_\beta$.}
\label{ss.rotationperturbations2}
For $\alpha$ and $\beta$ in $\mathbb{R}$, we define the \textit{rotating perturbations}
\begin{equation*}
\label{e.familiaS}
\Gamma^P_{\alpha}\colon M\to M,\quad \Gamma^P_{\alpha}\eqdef R^x_{\alpha,8}
\quad\mbox{and}\quad 
\Gamma^Q_{\beta}\colon M\to M,\quad 
\Gamma^Q_{\beta}\eqdef R^y_{\beta,8},
\end{equation*}
where 
$R^x_{{\omega,\rho}}$ and  $R^y_{{\omega,\rho}}$
are the rotation like-perturbations in \eqref{e.rotationfamily}.
As was observed, if $\alpha$ and $\beta$ are small enough the maps $\Gamma^P_{\alpha}$ and $\Gamma^Q_{\beta}$ are $C^r$-perturbations of the identity 
supported in $V_P$ and $V_Q$, respectively, 
such that
$$
\Gamma^P_{\alpha}(V_P)= V_P\quad \mbox{and}\quad 
\Gamma^Q_{\beta}(V_Q)= V_Q.
$$


%
%

\subsection{Properties of the unfolding family $\mathfrak{F}$}
\label{ss.propertiesoftheunfoldingfamily}
We now list some relevant properties 
satisfied by the family $\mathfrak{F}$.
\begin{remark}[Properties of $\mathfrak{F}$]
\label{r.propertiesoftheunfoldingfamily}
 {\em{
 By construction,
$f_{\bar{0}}=f$ and for every small (in norm) $\bar\upsilon=(\bar \mu, \bar \nu, \alpha, \beta)$  
we have that $f_{\bar\upsilon} $ is a diffeomorphism $C^r$-close to $f$
having a pair  of saddle points $P_{\bar\upsilon}=P$ and $Q_{\bar\upsilon}=Q$ such that (see Figure~\ref{fig:Unfolding}):
\begin{enumerate}
\item  
For every $W\in U_P\cap f^{-1}(V_P)$,
it holds $f_{\bar\upsilon}(W)=\Gamma^P_{{\alpha}}\circ f(W).
$
\item 
For every $W\in U_Q\cap f^{-1}(V_Q)$,
it holds $
f_{\bar\upsilon}(W)=\Gamma^Q_{\beta}\circ f(W).
$
\item 
 $f_{\bar \upsilon}\big(f^{-1}\big(\mathbb{B}_{\rho}(\widetilde{X})\big)\big)=\mathbb{B}_{\rho}(\widetilde{X})$
and 
$f^{N_1}_{\bar \upsilon}(X)=
f_{\bar \nu}\big(
f^{-1}(\widetilde X)\big)=\widetilde X+\bar \nu$.
\item 
 $f_{\bar \upsilon}\big(f^{-1}\big(\mathbb{B}_{\rho}(\widetilde{Y})\big)\big)=\mathbb{B}_{\rho}(\widetilde{Y})$
and 
$f^{N_2}_{\bar \upsilon}(X)=
f_{\bar \mu}\big(
f^{-1}(\widetilde Y)\big)=\widetilde Y+\bar \mu$.
\item 
$f_{\bar \upsilon}=f$ in 
 $M\setminus V_P\cup V_Q\cup f^{-1}\big(\mathbb{B}_{\rho}(\widetilde{X})\big)\cup f^{-1}\big(\mathbb{B}_{\rho}(\widetilde{Y})\big)$.
\end{enumerate}
}}
\end{remark}

%

\section{Return maps at the heterodimensional tangency} 
\label{s.ProofofTheoremt.teo1}
The renormalisation scheme of $f$ in Theorem~\ref{t.teo1} involves return maps (defined on a small 
neighbourhood of the heterodimensional tangency point $\widetilde{Y}$)
of the form
 \begin{equation}
\label{e.rttr}
F^{m,n}_{\bar \upsilon}\eqdef f^{N_2}_{\bar\upsilon}\circ f^{m}_{\bar\upsilon}\circ f^{N_1}_{\bar\upsilon}\circ f^{n}_{\bar\upsilon},
\end{equation}  
where $N_1$ and $N_2$ are the transition times between neighbourhoods of the saddles in the cycle in Conditions \textbf{B} and \textbf{C} (see Sections~\ref{ss.CB} and \ref{ss.CC}).  
For the next discussion recall the definitions of the 
quasi-transverse heteroclinic point $\widetilde{X}$ and of the
points $X\in U_Q$ with $f^{N_1}(X)= \widetilde{X} \in U_P$ and
$Y\in U_P$ with $f^{N_2}(Y)= \widetilde{Y} \in U_Q$. 
For suitable choices of $n$ and $m$, we are interested in 
iterations of points close to heterodimensional tangency point $\widetilde{Y}$ by the map $f_{\bar\upsilon}$ that after $n$ iterates in $U_Q$ land in a neighbourhood of
$X$. These points are mapped by $f^{N_1}_{\bar\upsilon}$ close to $\widetilde{X}$ in $U_P$. Thereafter, they remain in $U_P$ during $m$ 
iterates of $f_{\bar\upsilon}$, landing in a neighbourhood of $Y$. Finally, they return 
nearby $\widetilde{Y}$ by the transition $f^{N_2}_{\bar\upsilon}$.
The specially selected times 
$n$ and $m$  are called {\emph{sojourn times}}  in the neighbourhoods $U_P$ and $U_Q$. 
These times ``determine''  the set nearby $\widetilde{Y}$ where this return is defined.  

The composition \eqref{e.rttr} also demands another cares. For instance,  to guarantee  that after $m$ iterations 
certain points nearby $\widetilde X$
are mapped in  the domain of the transition from $P$ to $Q$ 
(i.e., in a small neighbourhood of $Y$)
it is necessary 
a small  change $\varphi_P+{\alpha}_{m}(\varphi_P)$ of
the  argument $\varphi_P$ of $Df(P)$, where ${\alpha}_{m}(\varphi_P)$ depends on the sojourn time $m$. 
  For the sojourn times $n$, we need to consider similar adjustments  ${\beta}_{n}(\varphi_Q)$ of
   the argument $\varphi_Q$ of $Df(Q)$. 
   The arguments $\varphi_P+{\alpha}_{m}(\varphi_P)$ and $\varphi_Q+{\beta}_{n}(\varphi_Q)$ are called \textit{adapted arguments}.
 
We divide the study of the renormalisation scheme in two parts:
(i) choice of adequate \textit{sojourn times} $m$ and $n$ and \textit{adapted arguments}, see  Sections~\ref{sss.sojourn} and~\ref{ss.adaptedarguments};
(ii) choice of suitable sequences of unfolding parameters $\bar \upsilon_{m,n}$ and charts $\Psi_{m,n}\colon \mathbb{R}^3\to M$, see Section~\ref{ss.elements}.  The convergence of the \textit{renormalised sequence}
$ \Psi_{m,n}^{-1} \circ F^{m,n}_{\bar \upsilon_{m,n}}\circ \Psi_{m,n}$ will be studied in  Sections~\ref{s.newProofofTheoremt.teo1}
and
\ref{ss.conclusion}.

\subsection{Sojourn times}
\label{sss.sojourn} 
Consider the set 
$$
\mathcal{Z}\eqdef
\{(\sigma,\lambda)\in\mathbb{R}^2: 0<\lambda<1<\sigma\}.
$$
An application of the next lemma will provide the desired sojourn times.
\begin{lemma}[Lemma 5.1 in \cite{DiaKirShi:14}]\label{l.neutraldynamics}
 There is a residual subset $\mathcal{R}$ of
$\mathcal{Z}$ consisting of points
$(\sigma,\lambda)$ satisfying the following property.
For every $\epsilon > 0$, $\xi > 0$, $\tau> 0$ and $N_0 > 0$, with 
$\epsilon< \xi $, there exist
integers $m, n > N_0$ such that 
$$
|\tau\sigma^m \lambda^n - \xi| < \epsilon \quad\mbox{and}\quad |m - n\eta - \tilde{\eta}| < 1,
$$
where 
$$
\eta = \frac{\log\lambda^{-1}}{\log\sigma}
\quad 
\mbox{and}
\quad
\tilde{\eta} = \frac{\log(\tau\xi^{-1})}{ \log\sigma}.
$$ 
In particular,
there exit sequences $(m_k), (n_k)\to +\infty$   as $k\to +\infty$
such that
$$
\lim_{k\to +\infty}
\,\sigma^{m_k}\,\lambda^{n_k}= \tau^{-1}\,\xi. 
$$
 \end{lemma}

\begin{remark}[Choice of sojourn times]\label{r.choiceofsojourn}
{\em{
A {\emph{sequence of sojourn times}} 
(adapted to $\sigma_P, \lambda_Q, \tau$ and $\xi$)
is any sequence $(\mathbf{s}_k)_k$, with  $\mathbf{s}_k= (m_k,n_k)\in \mathbb{N}^2$, 
obtained by applying Lemma~\ref{l.neutraldynamics} to: 
\begin{itemize}
\item
$(\sigma_P,\lambda_Q)\in \mathcal{Z}$, where $\lambda_Q,\sigma_P$ satisfying the 
spectral condition of the cycle in \eqref{e.seis}\footnote{As the spectral condition~\eqref{e.seis} is open, we can 
suppose  that $(\sigma_P,\lambda_Q)\in \mathcal{R}$.};
\item
${\tau}={\tau}(a_2,a_3,\gamma_3)\eqdef \dfrac{\gamma_3(a_3-a_2)}{\sqrt{2}}$,  where $a_2$, $a_3$ and $\gamma_3$ are 
the constants 
in the definition of the transition maps between $P$ and $Q$, see \eqref{e.transition1} and \eqref{e.transition2}.
Recall that by
\eqref{e.>} it holds ${\tau}(a_2,a_3,\gamma_3)>0$;
\item  $\xi>0$ is arbitrary but fixed. 
\end{itemize}
As a consequence, for a sequence of sojourn times  $\mathbf{s}_k = (m_k,n_k)$ it holds
\begin{equation}
\label{e.neutraldynamics}
\sigma_P^{m_k}\,\lambda_Q^{n_k}\rightarrow \tau^{-1}\xi=
\left(
\frac{\gamma_3(a_3-a_2)}{\sqrt{2}}
\right)
^{-1}\xi, \quad k\to+\infty.
\end{equation}
We say that the sequence $\bfs_k$ is {\emph{adapted to $\tau^{-1} \xi$.}}
}}
\end{remark}

\subsection{Adapted arguments}\label{ss.adaptedarguments}
We now discuss the choice of the ``adjusting arguments". 
Given a sequence $\mathbf{s}_k= (m_k,n_k)$ of sojourn times 
consider  a sequence $\Theta_k=(\zeta_{m_k}, \vartheta_{n_k})$ in $\mathbb{R}^2$ such that 
$\Theta_k\to (0,0)$ as $k\to+\infty$.
We call the pair 
$\fs_k=(\bfs_k,\Theta_k)$ a {\emph{sequence of sojourn times with associated arguments}}.
We define the {\emph{sequences 
${\alpha}^P_{k}= \alpha^P_{\fs_k}$ and  $\beta^Q_{k}= \beta^Q_{\fs_k}$
of argument adjustment maps}
as follows,
\begin{equation}
\begin{split}
\label{e.alphakbetak}
&{\alpha}^P_{k} :[-\pi,\pi]\to [-\pi,\pi],
\quad
{\alpha}^P_{k}(\theta) \eqdef
\frac1{2\pi m_k}\left(\frac{\pi}{4}-2\pi m_k\theta+
2\pi [m_k \theta]
+
\zeta_{m_k}
\right);
\\
&{\beta}^Q_{k}:[-\pi,\pi]\to [-\pi,\pi],
\quad 
{\beta}^Q_{k}(\omega) \eqdef
\frac1{2\pi n_k}
\left(
\frac{\pi}{2}-
2\pi n_k\omega+
2\pi [n_k\omega]+
\vartheta_{n_k}
\right),
\end{split}
\end{equation}
where  $[x]$ denotes the integer part
of $x\in \mathbb{R}$. 
Since $ x-[x]\in [0,1)$ it follows that 
$$
{\alpha}^P_{k}(\theta)\to 0,
\quad  
{\beta}^Q_{k}(\omega)\to 0,
\quad k\to+\infty,
$$
for every fixed $\theta,\omega\in [-\pi,\pi]$.

The  pair of sequences 
$$
\theta^P_{k}\eqdef
\theta+{\alpha}^P_{k}(\theta)\to \theta\quad\mbox{and}\quad  
\omega^Q_{k} \eqdef
\omega+{\beta}^Q_{k}(\omega)\to \omega,\quad k\to +\infty,
$$
is called  \textit{sequence of arguments adapted  to  $\theta$ and $\omega$} associated to $\fs_k=(\bfs_k, \Theta_k)$.

\subsection{Elements of the renormalisation scheme} 
\label{ss.elements}
Consider $0<\lambda_Q<1<\sigma_P$ satisfying~\eqref{e.seis} and $\xi>0$. If  
$(\sigma_P,\lambda_Q) \in \mathcal{Z}$, Lemma~\ref{l.neutraldynamics} provides
a sequence of sojourn times $\mathbf{s}_k =(m_k,n_k)\in\mathbb{N}^2$ adapted to $\tau^{-1} \xi$, where 
 $\tau$ is as Remark~\ref{r.choiceofsojourn}.
Consider a sequence of sojourn times with associated arguments
$\fs_k=(\bfs_k, \Theta_k)$, where
$\Theta_{k}=(\zeta_{m_k},\vartheta_{n_k})$.

 The
\textit{renormalisation scheme} $\mathcal{R}(\xi,\mathfrak{F})$  of $f$ 
consists of the following elements:
\begin{enumerate}
\item[(i)]
$\Psi_k=\Psi_{\mathbf{s}_k}\colon \mathbb{R}^3\to M$ a sequence of parameterisations 
 on the manifold $M$;
\item[(ii)]
$\bar \upsilon_k=\bar \upsilon_{\fs_k}:\mathbb{R}\times[-\pi,\pi]^2\to \mathbb{R}^8$ a sequence of bifurcation parameter maps of the family $\bar \upsilon \to f_{\bar \upsilon}$ in \eqref{e.thefamily}; 
\item[(iii)]
$\mathcal{R}_{k}(f)=\mathcal{R}_{\bar \upsilon_k}(f)$ is the sequence in $\diff^r(M)$ defined by
$$
f_{\bar \upsilon_k}^{N_2}\circ 
f_{\bar \upsilon_k}^
{m_k}\circ f_{\bar \upsilon_k}^{N_1}\circ f_{\bar \upsilon_k}^{n_k}.
$$
This composition is the \textit{renormalised sequence of $f$}.
\end{enumerate}

We  now give the precise definitions of the objects in the
renormalisation scheme.  In what follows, we fix  a sequence of sojourn times with arguments
$(\fs_k) =(\bfs_k, \Theta_k)$, where
$\bfs_k=(m_k,n_k)$ and 
$\Theta_k=(\theta_{m_k}, \vartheta_{n_k})$.
  
 \subsubsection{The parameterisations
  $\Psi_k=\Psi_{\bfs_k}$}\label{ss.parameterisations}
Given $\bfs= (m,n)\in \mathbb{N}^2$,  consider the map
$\Psi_{\bfs}\colon \mathbb{R}^3\to \mathbb{R}^3$ defined by
\[
\Psi_{\bfs}(x,y,z)\eqdef (1+\sigma_P^{-m}
\sigma_Q^{-n}\,x , \sigma_Q^{-n}+\sigma_P^{-2m}
{\sigma_Q}^{-2n}\,y ,1+ \sigma_P^{-m}
{\sigma_Q}^{-n}\,z ).
\]

\begin{remark}\label{r.convergencetoY}
{\em{
Let $K\subset \mathbb{R}^3$ be any compact set. The sequence of compact sets $(\Psi_{\bfs_k}(K))_k$ in $M$ satisfies
$$
\Psi_{\bfs_k}(K)\to \widetilde{Y} = (1, 0, 1) \in U_Q, \quad \mbox{as $k\to +\infty$},
$$
where the convergence is in the Hausdorff distance.  In particular,  there is $k_0=k_0(K)$ such that
$\Psi_{\bfs_k}(K) \subset U_Q$ for every $k\ge k_0$. In what follows, for notational simplicity, we write
$\Psi_{k}=\Psi_{\bfs_k}$. }}
\end{remark}

\subsubsection{The bifurcation parameters $\bar \upsilon_k=\bar \upsilon_{\fs_k}$} 
\label{ss.bifurcationparameters}
The sequence of maps $\bar \upsilon_{\fs_k}:\mathbb{R}^3\to \mathbb{R}^8$ is of the form
$$
\bar \upsilon_{\fs_k}(\mu,{\theta},\omega)\eqdef
\big(\bar{\mu}_{\bfs_k}(\mu),\bar\nu_{\fs_k},\alpha^P_{\fs_k}({\theta}),
\beta^Q_{\fs_k}(\omega)\big)
$$
where: 
\smallskip

\noindent
$\quad \bullet$
$\bar{\mu}_{\bfs_k} : \mathbb{R} \rightarrow \mathbb{R}^3$   
is defined by
\begin{equation}\label{e.muremu}
\bar{\mu}_{\bfs_k}(\mu)\eqdef(-\lambda_P^{m_k} a_1,\sigma_Q^{-n_k}+\sigma_Q^{-2n_k}
\sigma_P^{-2m_k}\mu -\lambda_P^{m_k} b_1, -\lambda_P^{m_k} c_1), 
\end{equation}
where $a_1,b_1,c_1$ are given in \eqref{e.transition2}. 
Note that 
$\bar{\mu}_{\bfs_k}(\mu)\to (0,0,0)$ as $k \to +\infty$ for every fixed $\mu \in \mathbb{R}$.
In what follows, for notational simplicity, we write
$\bar \mu_{k}=\bar{\mu}_{\bfs_k}$.
\smallskip

\noindent 
$\quad \bullet$
The maps
$\alpha^P_k(\theta)=\alpha^P_{\fs_k}(\theta)$ and $\beta^Q_k(\omega)=\beta^Q_{\fs_k}(\omega)$ 
are defined in \eqref{e.alphakbetak}.


\noindent
$\quad \bullet$
The sequence $\bar \nu_{\fs_k}\in \mathbb{R}^3$ is defined as follows. For the sequence of arguments adapted to $\varphi_{P}$ and $\varphi_{Q}$ (associated to $\fs_k$), we 
write 
$$
\varphi_{P,k}\eqdef (\varphi_{P})^P_k=
\varphi_P+{\alpha}^P_{k}(\varphi_P)\quad\mbox{and}\quad  
\varphi_{Q,k} \eqdef (\varphi_Q)^Q_{k}=
\varphi_Q+{\beta}^Q_{k}(\varphi_Q)
$$
and define the following sequences
\begin{equation}\label{e.fraksequences}
\begin{split}
\tilde{\mathfrak{c}}_k
&\eqdef\cos\big(
2\pi m_k (\varphi_{P,k})
\big),
\quad
\tilde{\mathfrak{s}}_k
\eqdef\sin\big(
2\pi m_k(\varphi_{P,k})
\big),
\\
\mathfrak{c}_k
&\eqdef\cos
\big(2\pi n_k(
\varphi_{Q,k}
)
\big),
\quad
\mathfrak{s}_k
\eqdef\sin\big(2\pi n_k
(\varphi_{Q,k})\big),
\end{split}
\end{equation}
and
\begin{equation}\label{e.removflatconditions}
\begin{split}
\tilde{\rho}_{2,k}\eqdef \frac1{2}\frac{\partial^2}{\partial x^2}\widetilde{H}_2(\textbf{0})(\mathfrak{c}_k-\mathfrak{s}_k)^2+\frac1{2}\frac{\partial^2}{\partial z^2}\widetilde{H}_2(\textbf{0})(\mathfrak{s}_k+\mathfrak{c}_k)^2,
\\
\tilde{\rho}_{3,k}\eqdef \frac1{2}\frac{\partial^2}{\partial x^2}\widetilde{H}_3(\textbf{0})(\mathfrak{s}_k-\mathfrak{c}_k)^2+\frac1{2}\frac{\partial^2}{\partial z^2}\widetilde{H}_3(\textbf{0})(\mathfrak{s}_k+\mathfrak{c}_k)^2.
\end{split}
\end{equation}
We now let

\begin{equation}\label{e.nu}
\begin{split}
\bar \nu_{\fs_k}&\eqdef
\Big(
-\lambda_Q^{n_k}\big(
\alpha_1\,(\mathfrak{c}_k-\mathfrak{s}_k)
+\alpha_3\,(\mathfrak{s}_k+\mathfrak{c}_k)\big),
\sigma_P^{-m_k}(\tilde{\mathfrak{c}}_k+\tilde{\mathfrak{s}}_k)-{\lambda_Q}^{2n_k}
\tilde{\rho}_{2,k},\\
& \qquad \qquad 
\sigma_P^{-m_k}(\tilde{\mathfrak{c}}_k-\tilde{\mathfrak{s}}_k)
-\lambda_Q^{n_k}\,\gamma_3\,(\mathfrak{c}_k+\mathfrak{s}_k)-\lambda_Q^{2n_k}
 \tilde{\rho}_{3,k}\Big).
\end{split}
\end{equation}
In what follows, for notational simplicity, we write
 $\bar \nu_k=\bar \nu_{\fs_k}$.

\begin{cl}
$\bar \nu_{k}\to (0,0,0)\in \mathbb{R}^3$ as $k \rightarrow +\infty$.
\end{cl}
\begin{proof}
Note that 
\begin{equation}
\label{e.limitscksk}
\begin{split}
&\tilde{\mathfrak{c}}_k=
\cos\Big(\frac{\pi}{4}
+\zeta_{m_k}
\Big)\to \frac{1}{\sqrt{2}}, 
\quad
\tilde{\mathfrak{s}}_k=\sin\Big(\frac{\pi}{4}
+
\zeta_{m_k}
\Big)
\to \frac{1}{\sqrt{2}},\\
&\mathfrak{c}_k=
\cos\left(\frac{\pi}{2}+\vartheta_{n_k}
\right)\to 0,\quad
\mathfrak{s}_k=
\sin\left(\frac{\pi}{2}+
\vartheta_{n_k}
\right)\to 1,
\end{split}
\end{equation}
when $k\to+\infty$.
Therefore,
$$
\tilde{\rho}_{2,k}\to \frac1{2}\frac{\partial^2}{\partial x^2}\widetilde{H}_2(\textbf{0})+
\frac1{2}\frac{\partial^2}{\partial z^2}
\widetilde{H}_2(\textbf{0}),
\quad 
\tilde{\rho}_{3,k}\to \frac1{2}\frac{\partial^2}{\partial x^2}\widetilde{H}_3(\textbf{0})+\frac1{2}\frac{\partial^2}{\partial z^2}\widetilde{H}_3(\textbf{0}).
$$
The claim  follows immediately from $|\lambda_P|, |\lambda_Q|<1$ and $|\sigma_P|,|\sigma_Q|>1$.
\end{proof}

\begin{remark}[Expression of  $f_{\bar\upsilon_{\fs_k} (\mu,\theta,\omega)}= f_{\bar\upsilon_{k} (\mu,\theta,\omega)} $]
\label{r.popov}
{\em{
Recall that 
$f_{\bar\upsilon} = \Gamma_{\bar \upsilon}\circ f$, see \eqref{e.thefamily}. Recalling the definitions of $\Gamma_{\bar{\upsilon}}$
in \eqref{e.definitionofGa}, of  $\Gamma^\q_{\bar \mu}$ and $\Gamma^\h_{\bar \nu}$ in Section~\ref{ss.unfoldingperturbation},
and of $\Gamma^P_{{\alpha}}$ and $\Gamma^Q_\beta$ in Section~
\ref{ss.rotationperturbations2},
we have that 
\begin{equation}
\label{e.renormitersysfunc}
f_{\bar \upsilon_{k} (\mu,\theta,\omega)} =
\left\{
\begin{split}
& \Gamma^P_{\alpha^P_k(\theta)} \circ f\eqdef  f_{\alpha^P_k(\theta)}, \quad \mbox{in $V_P$,}\\
& \Gamma^\h_{\bar \mu_{k}(\mu)} \circ f \eqdef
f_{\bar \mu_{k}(\mu)}, \quad \mbox{in $V_{f^{-1}(\widetilde Y)}$,} \\
&\Gamma^Q_{\beta^Q_k(\omega)} \circ f\eqdef f_{\beta^Q_k(\omega)}, \quad \mbox{in $V_Q$,}\\
&\Gamma^\q_{\bar \nu_{k}}\circ f\eqdef  f_{\bar \nu_{k}}, \quad \mbox{in $V_{f^{-1}(\widetilde X)}$,}\\
&f, \quad \mbox{in $M\setminus V_P\cup V_Q \cup V_{f^{-1}(\widetilde X)} \cup V_{f^{-1}(\widetilde Y)}$},
\end{split}
\right.
\end{equation}
where $V_{f^{-1}(\widetilde Z)}\eqdef f^{-1}\big(\mathbb{B}_{\rho}(\widetilde{Z})\big)$, $Z=X,Y$. See Remark \ref{r.propertiesoftheunfoldingfamily}.
}}
\end{remark}

\begin{remark}\label{r.Garbo}
{\em{
By definition of the arguments adapted to  $\varphi_P$ and $\varphi_Q$ (see Subsection \ref{ss.adaptedarguments}) 
it follows  that
the argument of $f_{\alpha_k(\varphi_P)}$ in the neighbourhoods $V_P$ of $U_P$ is given by
$
2\pi(\varphi_P+\alpha^P_k(\varphi_P))=2\pi(\varphi_{P,k})\to 2\pi(\varphi_{P})$. Thus, 
recalling  the local form of $f$ in the neighbourhoods of $P$ and $Q$ in equation \eqref{e.localdynamics},
$$
(f_{\alpha^P_k(\varphi_P)}|_{V_P})^{m_k}=\left( \begin{array}{ccc}
\lambda_P^{m_k} & 0 & 0 \\
0 & \sigma_P^{m_k}\tilde{\mathfrak{c}}_k & -\sigma_P^{m_k}
\tilde{\mathfrak{s}}_k\\
0 & \sigma_P^{m_k}
\tilde{\mathfrak{s}}_k & \sigma_P^{m_k}
\tilde{\mathfrak{c}}_k\end{array} \right),
$$
where $\tilde{\mathfrak{c}}_k, \tilde{\mathfrak{s}}_k$, are as in \eqref{e.fraksequences}.
Analogously, it holds 
$$
(f_{\beta^Q_k(\varphi_Q)}|_{V_Q})^{n_k}=\left( \begin{array}{ccc}
\lambda_Q^{n_k} {\mathfrak{c}}_k & 0 & -\lambda_Q^{n_k}  {\mathfrak{s}}_k\\
0 & \sigma_Q^{n_k}  & 0\\
\lambda_Q^{n_k}{\mathfrak{s}}_k & 0 & \lambda_Q^{n_k}{\mathfrak{c}}_k\end{array} \right).
$$
}}
\end{remark}

\begin{remark}[Convergence to $f$]
{\em{
By construction, 
given any $(\mu,\theta,\omega)\in\mathbb{R}^3$,
 the sequence 
$f_{\bar \upsilon_{k}, \Theta_k}$ with 
$\bar \upsilon_{k}=\bar \upsilon_{k}(\mu,\theta,\omega)$ converges to $f$ in the $C^r$-topology.
}}
\end{remark}


\section{The renormalised sequence of maps}
\label{s.newProofofTheoremt.teo1}
Fixed the
renormalisation scheme $\mathcal{R}(\xi,\mathfrak{F})$  of $f$ 
associated to the sequence $(\fs_k)=(\bfs_k,\Theta_k)$, where $\bfs_k=(m_k,n_k)$ and 
$\Theta_k=(\theta_{m_k}, \vartheta_{n_k})$,
and  a compact set $K\subset \mathbb{R}^3$, 
we now study the sequence of maps
$$
\Psi_{k}^{-1}\circ
\mathcal{R}_{\bar \upsilon_k}(f)
 \circ\Psi_{k}: K\subset \mathbb{R}^3 \rightarrow\mathbb{R}^3,\quad \bar \upsilon_k=\bar \upsilon_k (\mu, \theta, \omega).
 $$
Using the notation of the previous section, we begin by considering parameters of the form
$$
\bar \upsilon_{k}(\mu,\varphi_P,\varphi_Q)=
\big(\bar{\mu}_{k}(\mu),\bar\nu_{k},\alpha^P_{k}(\varphi_P),
\beta^Q_{k}(\varphi_Q)\big),
$$
For this choice of parameters,  the renormalisation scheme involves the adapted arguments  
$\varphi_P$ and $\varphi_Q$, the parameter $\bar \nu_k$ (which depends on the choice of these arguments,
see \eqref{e.nu}), and 
a ``free'' parameter $\mu$.
For notational simplicity, write 
$$
\alpha_{k,P}=\alpha^P_{k}(\varphi_P),\quad 
\beta_{k,Q}=\beta^Q_{k}(\varphi_Q),\quad  
\bar \upsilon_{k}(\mu) =
\bar \upsilon_{k}(\mu,\varphi_P,\varphi_Q),
$$
and for  $\bar{X}=(x,y,z)\in K$ let
\begin{equation}
\label{e.puntobreve}
\Psi_k^{-1} \circ \mathcal{R}_{\bar \upsilon_k (\mu)} \circ \Psi_k (x,y,z)=\breve X_k(x,y,z) \eqdef (\breve x_k, \breve y_k, \breve z_k).
\end{equation}
The goal of this section is to determine the coordinates of $\breve X_k$,
this is done in equations \eqref{e.finalbreve1}-\eqref{e.finalbreve3}. With these coordinates at hand, we will obtain the convergence of the renormalisation scheme.
 The  calculation of these coordinates involves three intermediate steps, corresponding to the compositions 
$\Psi_k (\bar X)$, 
$\mathcal{R}_{\bar \upsilon_k (\mu)} \circ \Psi_k (\bar X)$, and
$\Psi_k^{-1} \circ \mathcal{R}_{\bar \upsilon_k (\mu)} \circ \Psi_k (\bar X)$.

\subsection{Coordinates of $\Psi_k (\bar X)$}
Write
 $\bar{X}_k\eqdef  \Psi_k(\bar X)$ where  
\begin{equation}
\label{e.Jebsen}
\Psi_k(\bar X) =(1+\sigma_P^{-m_k}
{\sigma_Q}^{-n_k}\,x, \sigma_Q^{-n_k}+\sigma_P^{-2m_k}
\sigma_Q^{-2n_k}\,y, 1+ \sigma_P^{-m_k}
{\sigma_Q}^{-n_k}\,z).
\end{equation}
By the compactness of the set $K$, 
$\bar{X}_k\to \widetilde{Y} = (1, 0, 1)\in U_Q$, as $k\to+\infty$. Note that this step does not depend on $\mu$.

\subsection{Coordinates of  $\mathcal{R}_{\bar \upsilon_k (\mu)} \circ  \Psi_k (\bar X)$} Recall 
the definition of $f_{\bar \upsilon_{k} (\mu,\theta,\omega)}$ in \eqref{e.renormitersysfunc}.
The  application of  $\mathcal{R}_{\bar \upsilon_{k}(\mu)}$ involves the following four ``independent" intermediate
steps
(see Figure \ref{fig:pasos}): 
\begin{itemize}
\item{\bf{Step A:}} $n_k$ iterations nearby $Q$ given by
$f^{n_k}_{\bar \upsilon_k (\mu)}=f^{n_k}_{\beta_{k,Q}}$;
\item{\bf{Step B:}} the transition from $Q$ to $P$ along the orbit of $X$ given by  
$f^{N_1}_{\bar \upsilon_k (\mu)}=f_{\bar \nu_k} \circ f^{N_1-1}$;
\item{\bf{Step C:}} $m_k$ iterations nearby $P$ given by
$f^{m_k}_{\bar \upsilon_k (\mu)}=f^{m_k}_{\alpha_{k,P}}$;
\item{\bf{Step D:}} 
the transition from $P$ to $Q$ along the orbit of $Y$ given by  
$f^{N_2}_{\bar \upsilon_k (\mu)}=f_{\bar \mu_k(\mu)} \circ f^{N_2-1}$. This last step is the only one depending on  $\mu$.
\end{itemize}
\begin{figure}
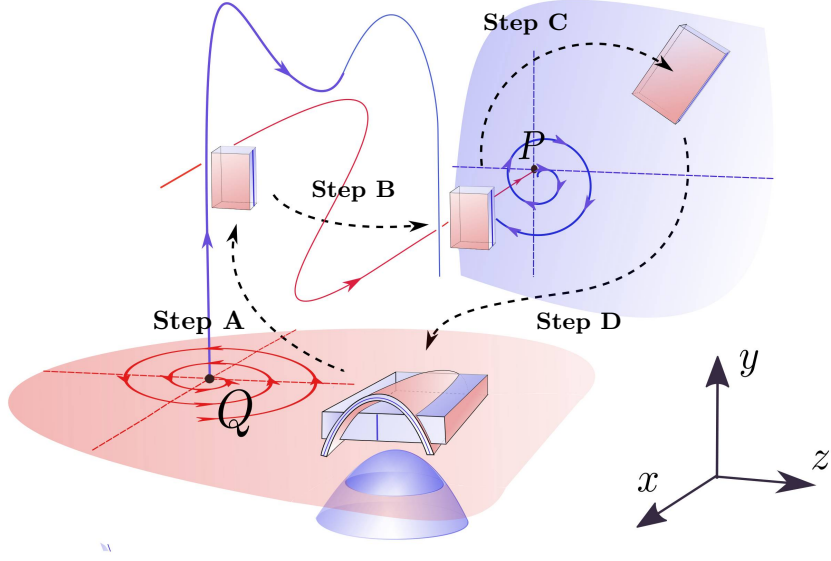

\centering
\begin{overpic}[scale=.13,
  ]{pasos2.jpg}
      \put(65,90){{$\textbf{Step A}$}}
         \put(125,140){\small{$\textbf{Step B}$}}
          \put(190,203){\small{$\textbf{Step C}$}}
              \put(210,90){\small{$\textbf{Step D}$}}
    \end{overpic}
\caption{Steps \textbf{A}-\textbf{D}.}
\label{fig:pasos}
\end{figure}
%

\noindent{\bf{Step A:}} Recalling \eqref{e.renormitersysfunc} and \eqref{e.Jebsen} , 
 write
\begin{equation}
\label{e.locdin1}
f^{n_k}_{\beta_{k,Q}} (\bar X_k)\eqdef 
(x_{k}, y_{k} + 1, z_{k})=X_k \in U_Q.
\end{equation}
The following equalities follow from
Remark~\ref{r.Garbo},
\begin{equation}
\label{e.locdin1bis}
\begin{split}
x_{k}& = {\sigma_P}^{-m_k}\,\lambda_Q^{n_k}\,\sigma_Q^{-n_k}\,(\mathfrak{c}_k\,x-\mathfrak{s}_k\,z)
+\lambda_Q^{n_k}\,(\mathfrak{c}_k-\mathfrak{s}_k),\\
y_{k}& = \sigma_P^{-2m_k}
\,{\sigma_Q}^{-n_k}\, y,\\
z_{k}& = {\sigma_P}^{-m_k}\,\lambda_Q^{n_k}\,\sigma_Q^{-n_k}\,(\mathfrak{s}_k\,x+\mathfrak{c}_k\,z)
+\lambda_Q^{n_k}(\mathfrak{c}_k+\mathfrak{s}_k).
\end{split}
\end{equation}
Since $K$ is a compact set,
we have that
\begin{equation}\label{e.simb1}
x_{k} = O(\lambda_Q^{n_k}),\quad y_{k} = O(\sigma_P^{-2m_k}\,{\sigma_Q}^{-n_k}),\quad z_{k} = O(\lambda_Q^{n_k}),
\end{equation}
where $O(\cdot)$ denotes the \textit{symbol of Landau}\footnote{Given two real valued functions $f$ and $g$,
one 
says that the \textit{symbol of Landau} of $f(x)$ is $O(g(x))$ (as $x\to+\infty$)
\textit{if and only if} there are $M>0$ and $x_0$ with
$|f(x)|\le M |g(x)|$ for all $x>x_0$.}.
Therefore
$$
X_k=(x_{k}, y_{k} + 1, z_{k})\to 
X=(0, 1, 0)\in U_X, \quad \mbox{as $k\to+\infty$.}
$$ 
\begin{equation}
\label{e.vector1}
\textbf{x}_k \eqdef (x_{k},y_{k},z_{k}).
\end{equation}}

\noindent{\bf{Step B:}}
The transition from a neighbourhood of $X$ to $U_P$ given by
  $f^{N_1}_{\bar \upsilon_k (\mu)}=f_{\bar \nu_k} \circ f^{N_1-1}$.
Write
\begin{equation}
\label{e.falta}
f_{\bar \nu_k} \circ f^{N_1-1}(X_k)\eqdef
(1+\tilde{x}_{k},\tilde{y}_{k},\tilde{z}_{k}) = \tilde X_k\in U_P.
\end{equation}
Using the definitions of $f^{N_1}$ in \eqref{e.transition1} (recall also \eqref{e.ddd}),
of $f_{\bar \nu_k}$ in \eqref{e.renormitersysfunc},
and of
$X_k$ in  \eqref{e.locdin1}, 
 we get
\begin{equation*}
\begin{split}
\tilde{x}_{k}&=
{\sigma_P}^{-m_k}\,\lambda_Q^{n_k}\,\sigma_Q^{-n_k}
\,\alpha_1\,(\mathfrak{c}_k\,x-\mathfrak{s}_k\,z) +
{\sigma_P}^{-2m_k}\,\sigma_Q^{-n_k}\,\alpha_2\,y
+
\\
&\qquad 
+
{\sigma_P}^{-m_k}\,\lambda_Q^{n_k}\,\sigma_Q^{-n_k}
\,\alpha_3\,(\mathfrak{s}_k\,x+\mathfrak{c}_k\,z)
+\widetilde{H}_1\big(\textbf{x}_k\big),
\\
\tilde{y}_{k}&={\sigma_P}^{-2m_k}\,\sigma_Q^{-n_k}\,\beta_2\,y+
\sigma_P^{-m_k}(\tilde{\mathfrak{c}}_k+\tilde{\mathfrak{s}}_k)
-
\lambda_Q^{2n_k}\tilde{\rho}_{2,k} +\widetilde{H}_2\big(\textbf{x}_k\big),
\\
\tilde{z}_{k}&=
{\sigma_P}^{-m_k}\,\lambda_Q^{n_k}\,\sigma_Q^{-n_k}\,\gamma_3\,(\mathfrak{s}_k\,x+\mathfrak{c}_k\,z)+
\sigma_P^{-m_k}(\tilde{\mathfrak{c}}_k-\tilde{\mathfrak{s}}_k)
-  \lambda_Q^{2n_k}\tilde{\rho}_{3,k} +\widetilde{H}_3\big
(\textbf{x}_k\big).
\end{split}
\end{equation*}
For  simplicity, in what follows, we write
\begin{eqnarray}\label{n.newhighordterm}
\widehat{H}_i\big(\textbf{x}_k\big) \eqdef
\widetilde{H}_i\big(\textbf{x}_k\big)
-\lambda_Q^{2n_k}
\tilde{\rho}_{i,k},\quad i=2,3.
\end{eqnarray}

\medskip

\noindent{\bf{Step C:}}
Recalling  \eqref{e.renormitersysfunc}, Remark~\eqref{r.Garbo} and  equation~\eqref{e.falta} writing
\begin{equation}
\label{e.ite3}
f^{m_k}_{\alpha_{k,P}} (\tilde X_k)\eqdef
(\hat{x}_{k},1+\hat{y}_{k},
1+\hat{z}_{k})= \hat X_k\in U_P,
\end{equation}
we have (after some straightforward simplifications and using that
$\mathfrak{s}_k^2+ \mathfrak{c}_k^2=1$) the following equalities:
\begin{equation*}
\begin{split}
\hat{x}_{k}&=
\lambda_P^{m_k}
+\lambda_P^{m_k}\,
{\sigma_P}^{-m_k}
\lambda_Q^{n_k}\,\sigma_Q^{-n_k}\,\big(
\alpha_1\,(\mathfrak{c}_k\,x-
\mathfrak{s}_k\,z)
+
\alpha_3\,
(\mathfrak{s}_k\,x+
\mathfrak{c}_k\,z)
\big)\\
&\qquad +
\lambda_P^{m_k} \sigma_P^{-2m_k}\,\sigma_Q^{-n_k}\,\alpha_2\,y
+{\lambda_P}^{m_k} 
\,
\widetilde{H}_1\big(\textbf{x}_k\big),
\\
\hat{y}_{k}&=\sigma_P^{-m_k}\,{\sigma_Q}^{-n_k}\tilde{\mathfrak{c}}_k\,\beta_2\,y-
\lambda_Q^{n_k}\,\sigma_Q^{-n_k}
\tilde{\mathfrak{s}}_k\,\gamma_3\,(\mathfrak{s}_k\,x+
\mathfrak{c}_k\,z) \\
&\qquad +
\sigma_P^{m_k}\,\Big(\tilde{\mathfrak{c}}_k\widehat{H}_2\big(\textbf{x}_k\big) -\tilde{\mathfrak{s}}_k\widehat{H}_3\big(\textbf{x}_k\big)\Big),
\\
\hat{z}_{k}&=\sigma_P^{-m_k}\,{\sigma_Q}^{-n_k}\tilde{\mathfrak{s}}_k\,\beta_2\,y+
\lambda_Q^{n_k}\,\sigma_Q^{-n_k}
\tilde{\mathfrak{c}}_k\,\gamma_3\,(\mathfrak{s}_k\,x+
\mathfrak{c}_k\,z)
\\
& \qquad +
{\sigma_P}^{m_k}\,\Big(\tilde{\mathfrak{s}}_k\widehat{H}_2\big(\textbf{x}_k\big) +\tilde{\mathfrak{c}}_k\widehat{H}_3\big(\textbf{x}_k\big)\Big).
\end{split}
\end{equation*}




\begin{lemma}\label{l.pausa}
$\hat{x}_{k}=O({\lambda_P}^{m_k})$,
$\hat{y}_{k}=O(\sigma_P^{-m_k}\,{\sigma_Q}^{-n_k})$, and  $\hat{z}_{k}=O(\sigma_P^{-m_k}\,{\sigma_Q}^{-n_k})$.
\end{lemma}

\begin{proof}
Note that
$\hat{y}_{k}$ and $\hat{z}_{k}$ have the same symbol of Landau.
The conditions in
\eqref{e.transition1} and \eqref{e.H_i} 
and the definition of
$\textbf{x}_k$ in \eqref{e.vector1} 
imply that the higher order terms
$\widetilde{H}_i$, $i=1,2,3$, 
are dominated by quadratic terms. Thus, for $i=1,2,3$, it holds

\begin{equation}
\begin{split}\label{e.taylorexpansion}
\widetilde{H}_i\big(\textbf{x}_{k}\big)&=
\frac1{2}\frac{\partial^2}{\partial x^2}\widetilde{H}_i(\textbf{0})\,x_{k}^2+
\frac{\partial^2}{\partial x\partial y}\widetilde{H}_i(\textbf{0})\,x_{k}\,y_{k}
+\frac{\partial^2}{\partial x\partial z}\widetilde{H}_i(\textbf{0})\,x_{k}\,z_{k}+\\
&\quad +
\frac1{2}\frac{\partial^2}{\partial y^2}\widetilde{H}_i(\textbf{0})\,y_{k}^2+\frac{\partial^2}{\partial y\partial z}\widetilde{H}_i(\textbf{0})\,y_{k}\,z_{k}+
\frac1{2}\frac{\partial^2}{\partial z^2}\widetilde{H}_i(\textbf{0})\,z_{k}^2+  \mathrm{h.o.t.}.
\end{split}
\end{equation}
Since $x_{k}$ and $z_{k}$ have the same Landau symbol $O(\lambda_Q^{2n_k})$ 
and $y_k$ has Landau symbol $O(\sigma_P^{-2m_k}\,{\sigma_Q}^{-n_k})$
(see \eqref{e.simb1}), it follows  that
\begin{equation}
\label{e.symbhightter}
\widetilde{H}_i\big(\textbf{x}_{k}\big)
=O(\lambda_Q^{2n_k})+
O({\sigma_P}^{-2m_k}\lambda_Q^{n_k}\sigma_Q^{-n_k})+O({\sigma_P}^{-4m_k}\,\sigma_Q^{-2 n_k}),\quad i=1,2,3.
\end{equation}
Finally, as the set $K$ is  compact,
it follows from \eqref{e.symbhightter} that $\hat{x}_{k}=O({\lambda_P}^{m_k})$. 

We now determine the symbol of $\hat{y}_{k}$ (and hence the one of  $\hat{z}_{k}$).
For this, its is enough to estimate the terms
${\sigma_P}^{m_k}
\widehat{H}_2
\big(\textbf{x}_k\big)$ and
${\sigma_P}^{m_k}
\widehat{H}_3
\big(\textbf{x}_k\big)$ (as the other terms in the expression of $\hat{y}_k$ are bigger).
Using the Taylor formula in \eqref{e.taylorexpansion},
the definition of the coordinates \eqref{e.locdin1bis},
and the definition of  $\tilde\rho_{i,k}$ in \eqref{e.removflatconditions},
it follows from \eqref{e.simb1} and \eqref{n.newhighordterm} that
\begin{equation*}
\begin{split}
\widehat{H}_i\big(\textbf{x}_k\big) &=
\widetilde{H}_i\big(\textbf{x}_{k}\big)
-\lambda_Q^{2n_k}\tilde{\rho}_{i,k}
=O({\sigma_P}^{-2m_k}\,\lambda_Q^{2n_k}\,\sigma_Q^{-2n_k})+O({\sigma_P}^{-m_k}\,\lambda_Q^{2n_k}\,\sigma_Q^{-n_k})+\\
&+
O({\sigma_P}^{-2m_k}\,\lambda_Q^{n_k}
\,\sigma_Q^{-n_k})+O(\sigma_P^{-4 m_k}\,{\sigma_Q}^{-2n_k}),
\quad i=2,3.
\end{split}
\end{equation*}

For the next estimate, recall that by \eqref{e.neutraldynamics}
 the sequence $\lambda_Q^{n_k}
\,\sigma_P^{m_k}$ converges to some number different from $0$.
Hence  $\lambda_Q^{n_k}
\,\sigma_P^{m_k}$ and
$\lambda_Q^{-n_k}
\,\sigma_P^{-m_k}$ are bounded from above by some $K_0>0$. 
This bound and $\sigma_P, \sigma_Q>1$ provide the
 following estimates,
\begin{itemize}
\item
${\sigma_P}^{-2m_k}\,\lambda_Q^{2n_k}\,\sigma_Q^{-2n_k}<{\sigma_P}^{-m_k}\,\lambda_Q^{2n_k}\,\sigma_Q^{-n_k}$,
\item
${\sigma_P}^{-2m_k}\,\lambda_Q^{n_k}
\,\sigma_Q^{-n_k}
=(\lambda_Q^{-n_k}
\,\sigma_P^{-m_k})\,{\sigma_P}^{-m_k}\,\lambda_Q^{2n_k}\,\sigma_Q^{-n_k}<K_0\,{\sigma_P}^{-m_k}\,\lambda_Q^{2n_k}\,\sigma_Q^{-n_k}$,
\item
$\sigma_P^{-4 m_k}\,{\sigma_Q}^{-2n_k}
=(\lambda_Q^{-2n_k}
\,\sigma_P^{-2m_k})(\sigma_P^{-m_k}
\,\sigma_Q^{-n_k})\,{\sigma_P}^{-m_k}\,\lambda_Q^{2n_k}\,\sigma_Q^{-n_k}$
\item[$\,$]\hspace{2.3cm}$< K_0^2 \, {\sigma_P}^{-m_k}\,\lambda_Q^{2n_k}\,\sigma_Q^{-n_k}$.
\end{itemize}
These inequalities  imply that  
\begin{equation}
\label{e.symbolnewhightter}
\widehat{H}_i\big(\textbf{x}_{k}\big)
=\widetilde{H}_i\big(\textbf{x}_{k}\big)
-\lambda_Q^{2n_k}\tilde{\rho}_{i,k}
=O({\sigma_P}^{-m_k}\,\lambda_Q^{2n_k}\,\sigma_Q^{-n_k}), \quad i=2,3.
\end{equation}
Therefore
$$
{\sigma_P}^{m_k}\widehat{H}_i
\big(\textbf{x}_{k}\big)=
O(\lambda_Q^{2n_k}\sigma_Q^{-n_k}), \quad i=2,3.
$$
Finally, 
observing that 
$$
\lambda_Q^{2n_k}\,\sigma_Q^{-n_k}=
(\sigma_P^{m_k}\,\lambda_Q^{n_k})\lambda_Q^{n_k}\,\sigma_P^{-m_k}\,{\sigma_Q}^{-n_k}<
K_0\,\sigma_P^{-m_k}\,{\sigma_Q}^{-n_k}
$$
we get 
$\hat{y}_{k} = O(\sigma_P^{-m_k}\,{\sigma_Q}^{-n_k})$, proving the lemma.
\end{proof}
Lemma \ref{l.pausa} implies that 
$$
\hat X_k=(\hat{x}_{k}, 1+\hat{y}_{k},1+ \hat{z}_{k})\to Y=(0,1,1)\in V_Y,
\quad k\to+\infty.
$$ 
For notational simplicity we write
\begin{equation}\label{e.kvector}
\hat{\textbf{x}}_k\eqdef (\hat{x}_{k},\hat{y}_{k},\hat{z}_{k}).
\end{equation}
\noindent{\bf{Step D:}} 
The transition from a neighbourhood of $Y$ to $U_Q$ is given by
 $f^{N_2}_{\bar \upsilon_k (\mu)}=f_{\bar \mu_k(\mu)} \circ f^{N_2-1}$.
Write 
\begin{equation}
\label{e.ddotk}
f_{\bar \mu_k(\mu)} \circ f^{N_2-1} (\hat X_k)\eqdef
(1+ \ddot{x}_{k},\ddot{y}_{k},1+\ddot{z}_{k})=\ddot{X_k}\in U_Q
\end{equation}
Using equation \eqref{e.transition2},
the expression of $f_{\bar \mu_k(\mu)} \circ f^{N_2-1}$ (see Remark~\ref{r.propertiesoftheunfoldingfamily}),
equation \eqref{e.renormitersysfunc}, the definition $\hat X_k$ in \eqref{e.ite3},
and the choice of
$\bar{\mu}_{k}(\mu)$
in \eqref {e.muremu} we have
\[
\begin{split}
\ddot{x}_{k}
&=
a_1\,{\lambda_P}^{m_k}
\,{\sigma_P}^{-m_k}\,
\lambda_Q^{n_k}\,\sigma_Q^{-n_k}\,
\big(
\alpha_1\,(\mathfrak{c}_k\,x-\mathfrak{s}_k\,z)
+\alpha_3\,(\mathfrak{s}_k\,x+\mathfrak{c}_k\,z)
\big)\\
&\quad+
a_1\,\lambda_P^{m_k}\,
\sigma_P^{-2m_k}
\,{\sigma_Q}^{-n_k}\,\alpha_2\,y
+
\big(\tilde{\mathfrak{c}}_k\,a_2+\tilde{\mathfrak{s}}_k\,a_3\big)\,\sigma_P^{-m_k}\,{\sigma_Q}^{-n_k}\,\beta_2\,y+\\
&\quad +
\lambda_Q^{n_k}\,\sigma_Q^{-n_k}\,\gamma_3
\,\big(\tilde{\mathfrak{c}}_k\,a_3-\tilde{\mathfrak{s}}_k\,a_2\big)\,
(\mathfrak{s}_k\,x+\mathfrak{c}_k\,z)+
\mathrm{h.o.t.}^*_{k},
\\
\ddot{y}_{k}&=
\sigma_Q^{-n_k}+\sigma_P^{-2m_k}\,{\sigma_Q}^{-2n_k}\mu +b_1\,\lambda_P^{m_k}\,
\sigma_P^{-2m_k}
\,{\sigma_Q}^{-n_k}\,\alpha_2\,y+
\\
&\quad +b_1\,{\lambda_P}^{m_k}\,
{\sigma_P}^{-m_k}\,
\lambda_Q^{n_k}\,\sigma_Q^{-n_k}\,
\big(
\alpha_1\,(\mathfrak{c}_k\,x-\mathfrak{s}_k\,z)
+\alpha_3\,(\mathfrak{s}_k\,x+\mathfrak{c}_k\,z)
\big)
+\\ 
&\quad+
\sigma_P^{-2m_k}\,
{\sigma_Q}^{-2n_k}
\Big(\tilde{\mathfrak{c}}_k^2\,b_2+\tilde{\mathfrak{s}}_k^2\,b_3+
\tilde{\mathfrak{c}}_k\,\tilde{\mathfrak{s}}_k
\,b_4\Big)\,\beta_2^2\,y^2 +
\\
&\quad +
\sigma_P^{-m_k}\,\lambda_Q^{n_k}\,\sigma_Q^{-2n_k}\,
\Big(
2\tilde{\mathfrak{c}}_k\,\tilde{\mathfrak{s}}_k\,(b_3-b_2)+ (\tilde{\mathfrak{c}}_k^2\,-\tilde{\mathfrak{s}}_k^2)\,b_4
\Big)\,\beta_2\,\gamma_3
\,
(\mathfrak{s}_k\,x\,y+\mathfrak{c}_k\,y\,z)
+\\
&\quad 
+\lambda_Q^{2n_k}\sigma_Q^{-2n_k}
\Big(
\tilde{\mathfrak{s}}_k^2\,b_2+\tilde{\mathfrak{c}}_k^2\,b_3-\tilde{\mathfrak{c}}_k\,\tilde{\mathfrak{s}}_k\,b_4
\Big)\,\gamma_3^2\,
(\mathfrak{s}_k\,x+\mathfrak{c}_k\,z)
^2+\mathrm{h.o.t.}^{**}_k,
\\
\ddot{z}_{k}&=
c_1\,{\lambda_P}^{m_k}
\,{\sigma_P}^{-m_k}\,
\lambda_Q^{n_k}\,\sigma_Q^{-n_k}\,
\big(
\alpha_1\,(\mathfrak{c}_k\,x-\mathfrak{s}_k\,z)
+\alpha_3\,(\mathfrak{s}_k\,x+\mathfrak{c}_k\,z)
\big)+  \\
&\quad+
c_1\,\lambda_P^{m_k}\,
\sigma_P^{-2m_k}
\,{\sigma_Q}^{-n_k}\,\alpha_2\,y
+
\big(\tilde{\mathfrak{c}}_k\,c_2+\tilde{\mathfrak{s}}_k\,c_3\big)\,\sigma_P^{-m_k}\,{\sigma_Q}^{-n_k}\,\beta_2\,y+\\
&\quad +
\lambda_Q^{n_k}\,\sigma_Q^{-n_k}\,\gamma_3
\,\big(\tilde{\mathfrak{c}}_k\,c_3-\tilde{\mathfrak{s}}_k\,c_2\big)\,
(\mathfrak{s}_k\,x+\mathfrak{c}_k\,z)+
\mathrm{h.o.t.}^{***}_{k},
\end{split}
\]
where
\begin{equation*}
\begin{split}
\mathrm{h.o.t.}^*_{k}&\eqdef
a_1\,{\lambda_P}^{m_k} 
\,
\widetilde{H}_1\big(\textbf{x}_k\big)+
a_2\,\sigma_P^{m_k}\,\Big(\tilde{\mathfrak{c}}_k\widehat{H}_2\big(\textbf{x}_k\big) -\tilde{\mathfrak{s}}_k\widehat{H}_3\big(\textbf{x}_k\big)\Big)+\\
&\quad+a_3\,\sigma_P^{m_k}\,\Big(\tilde{\mathfrak{s}}_k\widehat{H}_2\big(\textbf{x}_k\big) +\tilde{\mathfrak{c}}_k\widehat{H}_3\big(\textbf{x}_k\big)\Big)
+
H_1\big(\hat{\textbf{x}}_k\big),
\\
\mathrm{h.o.t.}^{**}_{k}&\eqdef
b_1\,{\lambda_P}^{m_k}
\widetilde{H}_1\big(\textbf{x}_k\big)+
b_2\,\Big[\sigma_P^{2m_k}\,\Big(\tilde{\mathfrak{c}}_k\widehat{H}_2\big(\textbf{x}_k\big) -\tilde{\mathfrak{s}}_k\widehat{H}_3\big(\textbf{x}_k\big)\Big)^2+\\
&\quad+
2\sigma_P^{m_k}\,\Big(\tilde{\mathfrak{c}}_k\widehat{H}_2\big(\textbf{x}_k\big) -\tilde{\mathfrak{s}}_k\widehat{H}_3\big(\textbf{x}_k\big)\Big)\Big(\sigma_P^{-m_k}\,{\sigma_Q}^{-n_k}\tilde{\mathfrak{c}}_k\,\beta_2\,y-
\lambda_Q^{n_k}\,\sigma_Q^{-n_k}
\tilde{\mathfrak{s}}_k\,\gamma_3\,
(\mathfrak{s}_k\,x+
\mathfrak{c}_k\,z)\Big)
\Big]+
\\
&\quad+
b_3\,\Big[\sigma_P^{2m_k}\,\Big(\tilde{\mathfrak{s}}_k\widehat{H}_2\big(\textbf{x}_k\big) +\tilde{\mathfrak{c}}_k\widehat{H}_3\big(\textbf{x}_k\big)\Big)^2+\\
&\quad+
2\sigma_P^{m_k}\,\Big(\tilde{\mathfrak{s}}_k\widehat{H}_2\big(\textbf{x}_k\big) +\tilde{\mathfrak{c}}_k\widehat{H}_3\big(\textbf{x}_k\big)\Big)
\Big(\sigma_P^{-m_k}\,{\sigma_Q}^{-n_k}\tilde{\mathfrak{s}}_k\,\beta_2\,y+
\lambda_Q^{n_k}\,\sigma_Q^{-n_k}
\tilde{\mathfrak{c}}_k\,\gamma_3\,
(\mathfrak{s}_k\,x+
\mathfrak{c}_k\,z)\Big)
\Big]\\
&\quad+
b_4\,\Big[\sigma_P^{2m_k}\,\Big(\tilde{\mathfrak{c}}_k\widehat{H}_2\big(\textbf{x}_k\big) -\tilde{\mathfrak{s}}_k\widehat{H}_3\big(\textbf{x}_k\big)\Big)\Big(\tilde{\mathfrak{s}}_k\widehat{H}_2\big(\textbf{x}_k\big) +\tilde{\mathfrak{c}}_k\widehat{H}_3\big(\textbf{x}_k\big)\Big)+\\
&\quad+
\sigma_P^{m_k}\,\Big(\tilde{\mathfrak{c}}_k\widehat{H}_2\big(\textbf{x}_k\big)-\tilde{\mathfrak{s}}_k\widehat{H}_3\big(\textbf{x}_k\big)\Big)
\Big(\sigma_P^{-m_k}\,{\sigma_Q}^{-n_k}\tilde{\mathfrak{s}}_k\,\beta_2\,y+
\lambda_Q^{n_k}\,\sigma_Q^{-n_k}
\tilde{\mathfrak{c}}_k\,\gamma_3\,
(\mathfrak{s}_k\,x+
\mathfrak{c}_k\,z)\Big)+
\\
&\quad+
\sigma_P^{m_k}\,\Big(\tilde{\mathfrak{s}}_k\widehat{H}_2\big(\textbf{x}_k\big)+\tilde{\mathfrak{c}}_k\widehat{H}_3\big(\textbf{x}_k\big)\Big)
\Big(\sigma_P^{-m_k}\,{\sigma_Q}^{-n_k}\tilde{\mathfrak{c}}_k\,\beta_2\,y-
\lambda_Q^{n_k}\,\sigma_Q^{-n_k}
\tilde{\mathfrak{s}}_k\,\gamma_3\,
(\mathfrak{s}_k\,x+
\mathfrak{c}_k\,z)\Big)
\Big]+\\
&\quad+
H_2\big(\hat{\textbf{x}}_k\big),
\\
\mathrm{h.o.t.}^{**\ast}_{k}&\eqdef c_1\,{\lambda_P}^{m_k} 
\,
\widetilde{H}_1\big(\textbf{x}_k\big)+
c_2\,\sigma_P^{m_k}\,\Big(\tilde{\mathfrak{c}}_k\widehat{H}_2\big(\textbf{x}_k\big) -\tilde{\mathfrak{s}}_k\widehat{H}_3\big(\textbf{x}_k\big)\Big)+\\
&\quad+c_3\,\sigma_P^{m_k}\,\Big(\tilde{\mathfrak{s}}_k\widehat{H}_2\big(\textbf{x}_k\big) +\tilde{\mathfrak{c}}_k\widehat{H}_3\big(\textbf{x}_k\big)\Big)
+
H_3\big(\hat{\textbf{x}}_k\big).
\end{split}
\end{equation*}
This concludes the computations in Steps \textbf{A} -\textbf{D}.
\subsection{Coordinates of $\Psi_k^{-1} \circ \mathcal{R}_{\bar \upsilon_k (\mu)} \circ  \Psi_k (\bar X)$}\label{ss.durruti}
To conclude the renormalisation calculations, it remains to apply  $\Psi_{k}^{-1}$ to the points $\ddot{X}_k$ in \eqref{e.ddotk}.
Note that, by the definition of $\breve X_k$ in \eqref{e.puntobreve} and by construction,
$$
\Psi^{-1}_k (\ddot X_k) = (\breve x_k, \breve y_k, \breve z_k) =\breve X_k.
$$
Recalling the expression of $\Psi_k$ in \eqref{e.Jebsen} we get
\begin{equation*}
{\Psi^{-1}_{k}}(1+\ddot{x_k},\ddot{y_k},1+\ddot{z_k})= (\sigma_P^{m_k}
{\sigma_Q}^{n_k}\ddot{x}, \sigma_P^{2m_k}\,{\sigma_Q}^{2n_k}(\ddot{y} -\sigma_Q^{-n_k}), \sigma_P^{m_k}
\,{\sigma_Q}^{n_k}\ddot{z}).
\end{equation*}
Applying the corresponding substitutions for $\ddot{x}_{k}$, $\ddot{y}_{k}$, and $\ddot{z}_{k}$,
we have:
\begin{equation}
\label{e.finalbreve1}
\begin{split}
&\breve{x}_k=
a_1\,\lambda_P^{m_k}\,
{\lambda_Q}^{n_k}\,
\big(
\alpha_1\,(\mathfrak{c}_k\,x-\mathfrak{s}_k\,z)
+\alpha_3\,(\mathfrak{s}_k\,x+\mathfrak{c}_k\,z)
\big)+\\
&\qquad +
a_1\,\lambda_P^{m_k}\,
\sigma_P^{-m_k}\,\alpha_2\,y
+
\big(\tilde{\mathfrak{c}}_k\,a_2+\tilde{\mathfrak{s}}_k\,a_3\big)\,\beta_2\,y+\\
&\qquad +
\sigma_P^{m_k}\,\lambda_Q^{n_k}\,\gamma_3
\,\big(\tilde{\mathfrak{c}}_k\,a_3-\tilde{\mathfrak{s}}_k\,a_2\big)\,
(\mathfrak{s}_k\,x+\mathfrak{c}_k\,z)+
\\
&\qquad +
{\sigma_P}^{m_k}\,\sigma_Q^{n_k}\,\mathrm{h.o.t.}^*_k,
\end{split}
\end{equation}
\begin{equation}
\label{e.finalbreve2}
\begin{split}
&\breve{y}_k
=
\mu +b_1\,\lambda_P^{m_k}
\,{\sigma_Q}^{n_k}\,\alpha_2\,y+
\\
&\qquad +b_1\,{\lambda_P}^{m_k}\,
{\sigma_P}^{m_k}\,
\lambda_Q^{n_k}\,\sigma_Q^{n_k}\,
\big(
\alpha_1\,(\mathfrak{c}_k\,x-\mathfrak{s}_k\,z)
+\alpha_3\,(\mathfrak{s}_k\,x+\mathfrak{c}_k\,z)
\big)
+\\ 
&\qquad+
\Big(\tilde{\mathfrak{c}}_k^2\,b_2+\tilde{\mathfrak{s}}_k^2\,b_3+
\tilde{\mathfrak{c}}_k\,\tilde{\mathfrak{s}}_k
\,b_4\Big)\,\beta_2^2\,y^2 +
\\
&\qquad 
+\sigma_P^{2m_k}\,\lambda_Q^{2n_k}
\Big(
\tilde{\mathfrak{s}}_k^2\,b_2+\tilde{\mathfrak{c}}_k^2\,b_3-\tilde{\mathfrak{c}}_k\,\tilde{\mathfrak{s}}_k\,b_4
\Big)\,\gamma_3^2\,
(\mathfrak{s}_k\,x+\mathfrak{c}_k\,z)
^2
+\\
&\qquad
+\sigma_P^{m_k}\,\lambda_Q^{n_k}\,
\Big(
2\tilde{\mathfrak{c}}_k\,\tilde{\mathfrak{s}}_k\,(b_3-b_2)+ (\tilde{\mathfrak{c}}_k^2\,-\tilde{\mathfrak{s}}_k^2)\,b_4
\Big)\,\beta_2\,\gamma_3
\,
(\mathfrak{s}_k\,x\,y+\mathfrak{c}_k\,y\,z)
+\\
&\qquad
+{\sigma_P}^{2m_k}\,\sigma_Q^{2n_k}\,\mathrm{h.o.t.}^{**}_k,
\end{split}
\end{equation}
\begin{equation}
\label{e.finalbreve3}
\begin{split}
&\breve{z}_k=
c_1\,\lambda_P^{m_k}\,
{\lambda_Q}^{n_k}\,
\big(
\alpha_1\,(\mathfrak{c}_k\,x-\mathfrak{s}_k\,z)
+\alpha_3\,(\mathfrak{s}_k\,x+\mathfrak{c}_k\,z)
\big)+\\
&\qquad +
c_1\,\lambda_P^{m_k}\,
\sigma_P^{-m_k}\,\alpha_2\,y
+
\big(\tilde{\mathfrak{c}}_k\,c_2+\tilde{\mathfrak{s}}_k\,c_3\big)\,\beta_2\,y+\\
&\qquad +\sigma_P^{m_k}\,\lambda_Q^{n_k}\,\gamma_3
\,\big(\tilde{\mathfrak{c}}_k\,c_3-\tilde{\mathfrak{s}}_k\,c_2\big)\,
(\mathfrak{s}_k\,x+\mathfrak{c}_k\,z)+
\\
&\qquad +
{\sigma_P}^{m_k}\,\sigma_Q^{n_k}\,\mathrm{h.o.t.}^{***}_k.
\end{split}
\end{equation}

This completes the calculations of $\Psi_k^{-1} \circ \mathcal{R}_{\bar \upsilon_k (\mu)} \circ  \Psi_k (\bar X)$.
%

\section{Convergence of the renormalised sequence: end of the proof of Theorem~\ref{t.teo1}.}\label{ss.conclusion}
Recall the choices of 
 $\xi>0$ and $\tau=\dfrac{\gamma_3(a_3-a_2)}{\sqrt{2}}>0$ in Remark \ref{r.choiceofsojourn} satisfying
$$
\sigma_P^{m_k}\,\lambda_Q^{n_k}\rightarrow \tau^{-1}\xi
, \quad k\to+\infty.
$$
Consider the vector $\bar \varsigma=\bar \varsigma(\xi,f)\eqdef (\varsigma_1,\varsigma_2,
\varsigma_3,\varsigma_4,\varsigma_5)\in\mathbb{R}^5$ where 
\begin{equation}
\begin{split}\label{e.2.24}
\varsigma_1&\eqdef\frac{\beta_2(a_2+a_3)}{\sqrt{2}},\quad
\varsigma_2\eqdef\frac{\beta_2^2(b_2+b_3+b_4)}{2},\quad
\varsigma_3\eqdef \xi^2\left(\frac{b_2+b_3-b_4}{(a_3-a_2)^{2}}\right),
\\
\varsigma_4&\eqdef \xi\sqrt{2}\left(\frac{\beta_2(b_3-b_2)}{a_3-a_2}\right),\quad
\varsigma_5\eqdef \frac{\beta_2(c_2+c_3)}{\sqrt{2}}.
\end{split}
\end{equation}
Note that the coordinates of $\bar \varsigma$ depend on  $\xi>0$ and the numbers $a_1,\dots, c_3$ in the definition of
  $f^{N_2}$ (see
\eqref{e.transition2}).

The next proposition implies Theorem~\ref{t.teo1}.

\begin{prop}  
\label{p.teo1}
Consider any compact set $K\subset \mathbb{R}^3$. Then
$$
\lim_{k\to+\infty}
\Big\Vert
\big(\Psi_k^{-1} \circ \mathcal{R}_{\bar \upsilon_k (\mu)} \circ  \Psi_k -E_{(\xi,\mu,\bar\varsigma(\xi,f))}\big)|_K
\Big\Vert_{C^r}=0,
$$
where $E_{(\xi,\mu,\bar\varsigma)}$ is 
the endomorphism in \eqref{e.E}.
\end{prop}


\begin{proof}
To make more transparent our calculations, let us first consider the leading terms (low order terms) of the coordinates
 $(\breve{x}_k, \breve{y}_k, \breve{z}_k)$ in \eqref{e.finalbreve1}-\eqref{e.finalbreve3}.
  Write  
  \[
 \breve{X}_k'\eqdef(\breve{x}_k', \breve{y}_k', \breve{z}_k'),
 \qquad
 \left\{
\begin{split}
\breve{x}_k'=
&\breve{x}_k-
\sigma_P^{m_k}\,{\sigma_Q}^{n_k}
\mathrm{h.o.t.}^*_k,
\\
\breve{y}_k'=
&\breve{y}_k-\sigma_P^{2m_k}\,{\sigma_Q}^{2n_k}
\mathrm{h.o.t.}^{**}_k,
\\
\breve{z}_k'= &\breve{z}_k-\sigma_P^{m_k}\,{\sigma_Q}^{n_k}
\mathrm{h.o.t.}^{***}_k.
\end{split}
\right.
\]
To prove the proposition it is enough to see that 
$$
 \breve{X}_k'\to E_{(\xi,\mu,\bar\varsigma)},
$$
and that 
$$
\sigma_P^{m_k}\,{\sigma_Q}^{n_k}
\mathrm{h.o.t.}^*_k\to 0,\quad
\sigma_P^{2m_k}\,{\sigma_Q}^{2n_k}
\mathrm{h.o.t.}^{**}_k\to 0,\quad  
\sigma_P^{m_k}\,{\sigma_Q}^{n_k}
\mathrm{h.o.t.}^{***}_k\to 0,
$$ 
where the 
converge occurs in the $C^r$-topology. This is done in the next two lemmas.

\begin{lemma}\label{l.buenaventura}
$
\lim_{k\to+\infty}
\Big\Vert
\big( \breve{X}_k' -E_{(\xi,\mu,\bar\varsigma(\xi,f))}\big)|_K
\Big\Vert_{C^r}=0.
$ 
\end{lemma}
\begin{proof}
This follows directly from
$\tilde{\mathfrak{c}}_k,\tilde{\mathfrak{s}}_k
\to \frac1{\sqrt{2}}$,
 $ \mathfrak{c}_k\to 0$, $\mathfrak{s}_k\to 1$ (see  \eqref{e.limitscksk}),
and 
$c_2=c_3$ (see \eqref{e.bbs}).
\end{proof}
\begin{lemma}
\label{l.hot}
The  terms $\sigma_P^{m_k}\,{\sigma_Q}^{n_k}
\mathrm{h.o.t.}^*_k$,
$\sigma_P^{2m_k}\,{\sigma_Q}^{2n_k}
\mathrm{h.o.t.}^{**}_k$, and  
$\sigma_P^{m_k}\,{\sigma_Q}^{n_k}
\mathrm{h.o.t.}^{***}_k$ 
converge to zero in the $C^r$-topology on compact sets.
\end{lemma}

\begin{proof} 
We begin with some preliminary observations. 
Since $\widehat{H}_2\big(\textbf{x}_k\big)$ and $\widehat{H}_3\big(\textbf{x}_k\big)$ have the same symbol of Landau  
$O(\sigma_P^{-m_k}\,\lambda_Q^{2n_k}\,\sigma_Q^{-n_k})$,
see \eqref{e.symbolnewhightter},  and
$\tilde{\mathfrak{c}}_k,\tilde{\mathfrak{s}}_k\to\tfrac1{\sqrt{2}}$, see \eqref{e.limitscksk},
the terms 
$$ 
\tilde{\mathfrak{c}}_k\widehat{H}_2\big(\textbf{x}_k\big) -\tilde{\mathfrak{s}}_k\widehat{H}_3\big(\textbf{x}_k\big),\quad 
\tilde{\mathfrak{s}}_k\widehat{H}_2\big(\textbf{x}_k\big) +\tilde{\mathfrak{c}}_k\widehat{H}_3\big(\textbf{x}_k\big)
$$
have both symbol of Landau\footnote{Here we use the following property:  if $g_1=O(f_1)$ and $g_2=O(f_2)$ then $g_1+g_2=O(|f_1|+|f_2|)$.} equal to 
$O(\sigma_P^{-m_k}\,\lambda_Q^{2n_k}\,\sigma_Q^{-n_k})$. With this in mind and using the multiplicative property of symbols of Landau\footnote{i.e., $O(f\cdot g)=O(f)\cdot O(g)$.},  we get that 
\begin{itemize}
 \item [(i)] 
${O} (\mathrm{h.o.t.}^*_k)= 
{O} \big( {\lambda_P}^{m_k} 
\,\widetilde{H}_1\big(\textbf{x}_k\big)  \big)+
{O}(
\sigma_P^{m_k}\,\widehat{H}_2\big(\textbf{x}_k\big))
+ {O}(
H_1\big(\hat{\textbf{x}}_k\big)).
$
\item[(ii)]
$O(\mathrm{h.o.t.}^{**}_k)= 
 O(
{\lambda_P}^{m_k})+ 
O\big(\widetilde{H}_1\big(\textbf{x}_k\big)\big)+
O\big(
\big(\sigma_P^{m_k}\,\widehat{H}_2\big(\textbf{x}_k\big)\big)^2\big)
+O\big(\sigma_Q^{-n_k}\,\widehat{H}_2\big(\textbf{x}_k\big)\big)
+ O \big(
H_2\big(\hat{\textbf{x}}_k\big) \big)$. In this estimate we use that $\sigma_P^{m_k}\,\lambda_Q^{n_k}$ is convergent. 
\item[(iii)]
$O(\mathrm{h.o.t.}^{***}_k)=
O\big(
{\lambda_P}^{m_k} 
\big( \widetilde{H}_1\big(\textbf{x}_k\big)\big)\big)+
O\big( \sigma_P^{m_k}\,\widehat{H}_2\big(\textbf{x}_k\big)\big)
+ O\big(
H_3\big(\hat{\textbf{x}}_k\big)\big).
$
\end{itemize}

Observing that $O\big(H_1\big(\hat{\textbf{x}}_k\big)\big)=O\big(H_3\big(\hat{\textbf{x}}_k\big)\big)$ have the same symbol of Landau (see \eqref{e.transition2}), the proof of  lemma is reduced to the following claim.
\begin{cl}\label{cl.reduction}
The following terms 
\begin{itemize}
\item[(a)]
$
{\lambda_P}^{m_k}\,\sigma_P^{2m_k}\,{\sigma_Q}^{2n_k} 
\,\widetilde{H}_1\big(\textbf{x}_k\big),
$
\item[(b)]
$
\sigma_P^{2m_k}\,{\sigma_Q}^{n_k}
\,\widehat{H}_2\big(\textbf{x}_k\big),
$
\item[(c)]
$\sigma_P^{m_k}\,{\sigma_Q}^{n_k} 
\,H_1\big(\hat{\textbf{x}}_k\big)$,
\item[(d)]
$\sigma_P^{2m_k}\,{\sigma_Q}^{2n_k} 
\,H_2\big(\hat{\textbf{x}}_k\big)$,
\end{itemize}
converge to zero in the $C^r$-topology on compact sets.
\end{cl}
\begin{proof}
We begin with some preliminary estimates.
Recalling the symbols of Landau of $\widetilde{H}_i\big(\textbf{x}_k\big)$
and 
 $\widehat{H}_i(\textbf{x}_k)$
in
\eqref{e.symbhightter} and  \eqref{e.symbolnewhightter}, respectively,
the
definition and properties of 
$H_i$ in \eqref{e.hs},
and the definition of  $\hat{\textbf{x}}_k$ in \eqref{e.kvector} whose coordinates satisfy Lemma \ref{l.pausa},
it follows that
\begin{equation}\label{e.calice}
\begin{split}
\widetilde{H}_1\big(\textbf{x}_k\big)
&=O(\lambda_Q^{2n_k})+O({\sigma_P}^{-2m_k}\,\lambda_Q^{n_k}
\sigma_Q^{-n_k})+
O(\sigma_P^{-4m_k}\,{\sigma_Q}^{-2n_k}),\\
\widehat{H}_2\big(\textbf{x}_{k}\big)
&=O({\sigma_P}^{-m_k}\,\lambda_Q^{2n_k}\,\sigma_Q^{-n_k}),\\
H_1\big(\hat{\textbf{x}}_k\big)&=
O({\lambda_P}^{2m_k})
+O({\lambda_P}^{m_k}\sigma_P^{-m_k}{\sigma_Q}^{-n_k})
+
O(\sigma_P^{-2m_k}
\,{\sigma_Q}^{-2n_k}),
\\
H_2\big(\hat{\textbf{x}}_k\big)&=
O({\lambda_P}^{2m_k})
+O({\lambda_P}^{m_k}\sigma_P^{-m_k}{\sigma_Q}^{-n_k}).
\end{split}
\end{equation}
To prove (a),  note that
\begin{equation}\label{e.porto}
{\lambda_P}^{m_k}
\sigma_P^{2m_k}
{\sigma_Q}^{2n_k}\to 0,\quad k\to+\infty.
\end{equation}
 Lemma \ref{l.neutraldynamics} provides the constant $C\eqdef ({\lambda_P}^{\frac1{2}}
{\sigma_P})^{2(1+\tilde{\eta})}>0$ such that 
\[
{\lambda_P}^{m_k}
\sigma_P^{2m_k}
{\sigma_Q}^{2n_k}=
\big(
{\lambda_P}^{\frac{m_k}{2}}{\sigma_P}^{m_k}\sigma_Q^{n_k}
\big)^2< C\,
\Big(\big(({\lambda_P}^{\frac1{2}}
{\sigma_P})^{\eta}
\sigma_Q\big)^{n_k}\Big)^2,
\]
where
$$
\eta=\dfrac{\log\lambda_Q^{-1}}{\log{\sigma_P}}\quad\mbox{and}\quad 
\tilde{\eta}=\dfrac{\log{(\tau \,\xi^{-1})}}{\log{\sigma_P}}.
$$
The convergence ${\lambda_P}^{m_k}
\sigma_P^{2m_k}
{\sigma_Q}^{2n_k}\to 0$ as $k\to+\infty$ follows from 
$0<({\lambda_P}^{\frac1{2}}\,{\sigma_P})^{\eta}\sigma_Q <1$ (see \eqref{e.seis}).
Thus, on any compact set $K$ of $\mathbb{R}^3$, we have the convergence
$$
{\lambda_P}^{m_k}\,\sigma_P^{2m_k}\,{\sigma_Q}^{2n_k} 
\,\widetilde{H}_1\big(\textbf{x}_k\big)
\to 0,\quad k\to +\infty.
$$
It is easy to see that this convergence 
 also holds for  the derivatives of order $1\le k\le r$.
Therefore 
$$
\lim_{k\to+\infty}
\Big\Vert 
{\lambda_P}^{m_k}\,\sigma_P^{2m_k}\,{\sigma_Q}^{2n_k} 
\,\widetilde{H}_1\big(\textbf{x}_k\big)
|_K
\Big\Vert_{C^r}=0.
$$

To prove (b), observe that  \eqref{e.calice} and \eqref{e.neutraldynamics} imply that
$$
\sigma_P^{2m_k}\,\sigma_Q^{n_k}\,\widehat{H}_2\big(\textbf{x}_{k}\big)
=O(\sigma_P^{m_k}\,\lambda_Q^{2n_k})=O(\lambda_Q^{n_k})\to 0,\quad k\to+\infty.
$$
Arguing as in (a) we get 
$$
\lim_{k\to+\infty}
\Big\Vert 
\sigma_P^{2m_k}\,\sigma_Q^{n_k}\,\widehat{H}_2\big(\textbf{x}_{k}\big)
|_K
\Big\Vert_{C^r}=0.
$$

To prove (c), observe that \eqref{e.calice} and \eqref{e.porto} imply that
\begin{equation*}
\begin{split}
\sigma_P^{m_k}\,\sigma_Q^{n_k}\,H_1\big(\hat{\textbf{x}}_k\big)&=
O({\lambda_P}^{2m_k}\,\sigma_P^{m_k}\,\sigma_Q^{n_k})
+O({\lambda_P}^{m_k})
+O(\sigma_P^{-m_k}
\,{\sigma_Q}^{-n_k})\\
&\le 
O({\lambda_P}^{m_k})\cdot O({\lambda_P}^{m_k}\,\sigma_P^{2m_k}\,\sigma_Q^{2n_k})\\
&+O({\lambda_P}^{m_k})
+O(\sigma_P^{-m_k}
\,{\sigma_Q}^{-n_k})
\to 0, \quad k\to+\infty.
\end{split}
\end{equation*}
Thus, arguing as in the previous cases,
$$
\lim_{k\to+\infty}
\Big\Vert 
\sigma_P^{m_k}\,{\sigma_Q}^{n_k} 
\,H_1\big(\hat{\textbf{x}}_k\big)
|_K
\Big\Vert_{C^r}=0.
$$

Finally, to prove (d), observe that \eqref{e.calice} and \eqref{e.porto} imply that
\begin{equation*}
\begin{split}
\sigma_P^{2m_k}\,\sigma_Q^{2n_k}\,H_2\big(\hat{\textbf{x}}_k\big)&=
O({\lambda_P}^{2m_k}\,\sigma_P^{2m_k}\,\sigma_Q^{2n_k})
+O({\lambda_P}^{m_k}\sigma_P^{m_k}{\sigma_Q}^{n_k})\\
&\le
O({\lambda_P}^{m_k})\cdot O({\lambda_P}^{m_k}\,\sigma_P^{2m_k}\,\sigma_Q^{2n_k})\\
&+O({\lambda_P}^{m_k}\sigma_P^{2m_k}{\sigma_Q}^{2n_k})\to 0.
\end{split}
\end{equation*}
Thus, as in the previous cases,
$$
\lim_{k\to+\infty}
\Big\Vert 
\sigma_P^{2m_k}\,{\sigma_Q}^{2n_k} 
\,H_2\big(\hat{\textbf{x}}_k\big)
|_K
\Big\Vert_{C^r}=0.
$$

This proves the claim.
\end{proof}
The proof or the  lemma is now complete.
\end{proof}
The proof of the proposition is now complete.
\end{proof}
This completes the proof of Proposition \ref{p.teo1}.

\section{Non-transverse cycles leading blender-horseshoes} \label{ss.ultima}
We  discuss here the existence of diffeomorphisms $f$  satisfying the hypotheses of Corollary~\ref{c.c1}.

\begin{lemma}\label{l.ultimolema}
There are diffeomorphisms $f$  satisfying conditions $\mathbf{(A)}$-$\mathbf{(C)}$
whose vector 
$\bar \varsigma=\bar\varsigma (f,\xi)$ satisfies equation \eqref{e.restrictions} for every  $\xi \in (1.18,1.19)$.
\end{lemma}

\begin{proof}
Consider the vector
$$
v=(a_1,a_2,a_3,,b_1,b_2,b_3,b_4,c_1,c_2, c_3)\in \mathbb{R}^{10},
$$
where $a_1,\dots, c_3$ are the constants
defining the transition map from $P$ to $Q$ in \eqref{e.transition2}. 
The precise formula for the vector $\bar \varsigma$ is given in equation \eqref{e.2.24} and 
is of the form  (with a slight abuse of notation)
$\bar \sigma (f,\xi)=\bar\varsigma (v,\xi)= (\bar\varsigma_1 (v,\xi), \dots, \bar\varsigma_5 (v,\xi))$,
where we interpret the coeficients $\bar\varsigma_i$ as maps depending on $(v,\xi)$.

Consider the family of maps   $\gamma_\xi \colon \mathbb{R}^{11} \to \mathbb{R}^2$ defined by
$$
\gamma_{\xi}: \mathbb{R}^{11}\to(-\varepsilon,\varepsilon)^2,\quad \gamma_{\xi}(v)\eqdef \big(\kappa(v,\xi), \eta(v,\xi)\big),
$$
where
$\kappa, \eta \colon \mathbb{R}^{11} \to \mathbb{R}$ are given by
$$
\kappa(v,\xi)\eqdef \frac{\varsigma_1^2 (v,\xi) \varsigma_3 (v,\xi)}{ \varsigma_2(v,\xi)}, \qquad
\eta(v,\xi)\eqdef
\frac{\varsigma_1 (v,\xi) \varsigma_5 (v,\xi)}{ \varsigma_2(v,\xi)}.
$$
It is immediate to verify that given any $\xi \in (1.18,1.19)$ then
every
 $(\kappa_0,\eta_0)\in(-\varepsilon,\varepsilon)^2$ is a regular value of  
$\gamma_{\xi}$. This implies the lemma.
\end{proof}

\bibliographystyle{siam}

\end{document}